\documentclass[12pt]{amsart}
\usepackage[dvips]{epsfig}
\usepackage{graphics}
\usepackage{latexsym}
\usepackage{verbatim}
\usepackage{amsmath}
\usepackage{amsthm}
\usepackage{amssymb}
\usepackage{hyperref}
\newtheorem{theorem}{Theorem}[section]
\newtheorem{lemma}[theorem]{Lemma}
\newtheorem{corollary}[theorem]{Corollary}
\newtheorem{proposition}[theorem]{Proposition}

\def\Remark{\medskip\noindent{\bf Remark: }}
\def\Remarks{\medskip\noindent{\bf Remarks: }}

\newcommand{\ens}[1]{\mathbb{#1}}

\newcommand{\N}{\mathbb{N}}
\newcommand{\NN}{\mathbb{N}}
\newcommand{\Z}{\mathbb{Z}}
\newcommand{\R}{\mathbb{R}}

\newcommand{\C}{\mathcal{C}}
\newcommand{\PP}{\mathbb{P}}

\def\cal{\mathcal}
\def\supp{\mbox{supp }}

\def\derpar#1#2{\frac{\partial#1}{\partial#2}}
\def\var{\varepsilon}

\begin{document}

\title[Analysis of spectral methods for the Boltzmann equation]
{Analysis of spectral methods for the homogeneous Boltzmann equation}

\author{Francis Filbet and Cl\'ement Mouhot}

\hyphenation{bounda-ry rea-so-na-ble be-ha-vior pro-per-ties
cha-rac-te-ris-tic}

\maketitle

\begin{abstract} 
  The development of accurate and fast algorithms for the Boltzmann
  collision integral and their analysis represent a challenging
  problem in scientific computing and numerical analysis.  Recently,
  several works were devoted to the derivation of spectrally accurate
  schemes for the Boltzmann equation, but very few of them were
  concerned with the stability analysis of the method. In particular
  there was no result of stability except when the method is modified
  in order to enforce the posivity preservation, which destroys the
  spectral accuracy. In this paper we propose a new method to study
  the stability of homogeneous Boltzmann equations perturbed by
  smoothed balanced operators which do not preserve positivity of the
  distribution. This method takes advantage of the ``spreading''
  property of the collision, together with estimates on regularity and
  entropy production. As an application we prove stability and
  convergence of spectral methods for the Boltzmann equation, when the
  discretization parameter is large enough (with explicit bound).
\end{abstract}


\medskip


\tableofcontents

\section{Introduction}\label{sec:intro}
\setcounter{equation}{0} This work is the sequel of several papers
devoted to the approximation of the Boltzmann equation using fast
spectral methods \cite{PaRu:spec:00,FiRu:FBE:03,MP:03, FMP}. The
present paper is devoted to the stability and convergence analysis of
general spectral algorithms.
 
In a microscopic description of rarefied gas, the particles move by a
constant velocity until they undergo binary collisions. In statistical
physics, the properties of the gas are described by a density function
in phase space, $f(t,x,v)$, called the {\em distribution function},
which gives the fraction of particles per unit volume in phase space
at time $t$. The distribution function satisfies the Boltzmann
equation, a non-linear integro-differential equation, which describes
the combined effect of the free flow and binary collisions between the
particles.

The main difficulty in the approximation of the Boltzmann equation is
due to the multidimensional structure of the collisional integral,
since the integration runs on a highly-dimensional unflat manifold. In
addition the numerical integration requires great care since the
collision integral is at the basis of the macroscopic properties of
the equation. Further difficulties are represented by the presence of
stiffness, like the case of small mean free path or the case of large
velocities~\cite{FiPa:02}.

For such reasons realistic numerical simulations are based on
Monte-Carlo techniques. The most famous examples are the {Direct
Simulation Monte-Carlo (DSMC)} methods by Bird~\cite{bird} and by
Nanbu~\cite{Na}. These methods guarantee efficiency and
preservation of the main physical properties. However, avoiding
statistical fluctuations in the results becomes extremely
expensive in presence of non-stationary flows or close to
continuum regimes.

Among deterministic approximations, Discrete Velocity Models (DVM) are
based on a cartesian grid in velocity and on a discrete collision
mechanism on the points of the grid that preserves the main physical
properties. Unfortunately DVM are not competitive with Monte-Carlo
methods in terms of computational cost and their accuracy seems to be
less than first
order~\cite{BoPaSc:cons:97,PaSc:stabcvgDVM:98,HePa:DVM:02,Wenn2d}.

Another class of numerical methods, based on the use of spectral
techniques in the velocity space, has been developed. The methods were
first derived in~\cite{PePa:96}, inspired from spectral methods in
fluid mechanics~\cite{CHQ:88} and by previous works on the use of
Fourier transform techniques for the Boltzmann equation
\cite{Boby:88}.  They are based on approximating in the velocity space
the distribution function by a periodic function, and on its
representation by Fourier series.
% The resulting Fourier-Galerkin approximation can be evaluated with a
% computational cost of $O(n^{2})$ (where $n$ is the total number of
% discretization parameters in velocity), which is lower than that of
% previous deterministic methods (but still larger than that of
% Monte-Carlo methods).
The spectral method has been further developed in
\cite{PaRu:spec:00,PaRu:stab:00} where evolution equations for the
Fourier modes were explicitly derived and spectral accuracy of the
method was proven.  Strictly speaking these methods are not
conservative, since they only preserve mass, whereas momentum and
energy are approximated with spectral accuracy.  Moreover, the
spectral method has been applied also to non homogeneous
situations~\cite{FiRu:FBE:03}, to the Landau equation~\cite{FiPa:02,
  PaRuTo:00}, where fast algorithms can be readily derived, and to the
case of granular gases~\cite{FiPa:03}.  Independantly A.~Bobylev \&
S.~Rjasanow \cite{BoRj:HS:97,BoRj:HS:99} have also constructed fast
algorithms based on a Fourier transform approximation of the
distribution function, but the method is not spectrally accurate (only
second order).

In \cite{MP:03} a fast spectral method was proposed for a class of
particle interactions including pseudo-Maxwell molecules in dimension
$2$ and, most importantly, hard spheres in dimension $3$, on the basis
of the previous spectral method together with a suitable
semi-discretization of the collision operator. This method permits to
reduce the computational cost from $O(n^2)$ to $O(n \log_2 n)$ without
loosing the spectral accuracy, thus making the method competitive with
Monte-Carlo.

However an important drawback of the spectral methods up to now had
been the lack of proof of stability. Indeed as compared to discrete
velocity methods the difficulties are somehow opposite: consistency
results are easily obtained, whereas the lack of positivity
preservation of the scheme is a major issue when one studies its
stability properties. The only paper concerned with the issue of
stability for spectral methods applied to the Boltzmann collision
operator is \cite{PaRu:stab:00}, but in the latter the author
introduce some ``filters'' on the Fourier modes in order to restore
the posivity-preservation of the scheme, which breaks the spectral
accuracy.

In this paper we give the first stability result for the spectral
methods applied to the Boltzmann collision operator. Moreover we
propose a method which is likely to have other utilizations in
collisional kinetic theory:
\begin{itemize} 
\item we write the Galerkin approximation on the first $N$ Fourier
  modes of the evolution equation as a {\em a smooth balanced
    perturbation} of the original equation, in the sense of a
  perturbation by some small and mass-preservation (although not
  positivity-preserving) error term;
\item we prove existence and uniqueness of smooth solution for small
  times, conditionally to a bound on the $L^1$ norm;
\item we use the mixing structure \cite{a44,M:04} of the collision
  process to show appearance of positivity after a small time
  (depending on the size of the box of truncation and the
  approximation parameter $N$);
\item we use the mass conservation to deduce uniform bounds on the
  $L^1$ norm, and therefore regularity bounds growing at most
  exponentially in time;
\item we perform a detailed analysis of the unperturbed truncated
  problem, showing uniform in time regularity and asymptotic
  convergence to equilibrium;
\item finally we use that the equilibrium is unchanged by the smooth
  balanced perturbation, and that it is non-linearly stable for the
  perturbed periodized Boltzmann equation, in order to prove global in
  time stability and convergence to equilibrium for the perturbed
  Boltzmann equation (we connect the previous point for initial times
  together with the stability of equilibrium for asymptotic times).
\end{itemize} 

Hence our paper introduce a general method on how to exploit fine
mixing properties of the collision process in the study of stability
of a particular class of perturbed Boltzmann equation, with the
application in mind to the stability of spectral methods.

The outline of this paper is as follows. The Boltzmann equation and
its basic features are presented in Section~\ref{sec2}. In
Section~\ref{sec3} we explain the truncation and periodization
associated with spectral methods and fast spectral methods and we
formulate the problem of stability of these methods in the general
framework of the stability properties of the Boltzmann equation with
respect to a smooth balanced perturbation. Section~\ref{sec4} is
devoted to the proof of the main stability result in the general
framework.  Section~\ref{sec5} is devoted to the study of the
asymptotic behavior of the truncated problem on the basis of the
entropy production theory. Finally in Section~\ref{sec6} we apply the
latter result to the spectral method and establish some stability and
convergence results of the numerical solution.

\section{The Boltzmann equation}\label{sec2}

The Boltzmann equation describes the behavior of a dilute gas of
particles when the only interactions taken into account are binary
elastic collisions. It reads for $x \in \Omega$, $v \in \R^d$ where
$\Omega \in \R^d$ is the spatial domain ($d \ge 2$):
 \begin{equation*}
 \derpar{f}{t} + v \cdot \nabla_x f = Q(f,f)
 \end{equation*}
where $f:=f(t,x,v)$ is the time-dependent particles distribution
function in the phase space. The Boltzmann collision operator $Q$
is a quadratic operator local in $(t,x)$. The time $t$ and position $x$
act only as parameters in $Q$ and therefore will be omitted in
its description
 \begin{equation} \label{eq:Q}
 Q (f,f)(v) = \int_{v_\star \in \R^d}
 \int_{\sigma \in \ens{S}^{d-1}}  B(|v-v_\star|, \cos \theta) \,
 \left( f^\prime_\star f^\prime - f_\star f \right) \, d\sigma \, dv_\star.
 \end{equation}
%In~\eqref{eq:Q} we dropped the $x$ variable and 
We used the shorthand $f = f(v)$, $f_\star = f(v_\star)$,
$f^\prime = f(v^\prime)$, $f_\star^\prime = f(v_\star^\prime)$. The
velocities of the colliding pairs $(v,v_\star)$ and
$(v^\prime,v^\prime_\star)$ are related by
 \begin{equation*}
%\label{eq:rel:vit}
\left\{
\begin{array}{l}
\displaystyle{ v^\prime   = v - \frac{1}{2} \big((v-v_\star) 
- |v-v_\star|\,\sigma\big),} 
\vspace{0.2cm} \\
\displaystyle{ v^\prime_\star = v - \frac{1}{2} \big((v-v_\star) + |v-v_\star|\,\sigma\big),}
\end{array}\right.
 \end{equation*}
 with $\sigma\in \ens{S}^{d-1}$.  The collision kernel $B$ is a
 non-negative function which by physical arguments of invariance only
 depends on $|v-v_\star|$ and $\cos \theta = {u} \cdot \sigma$ (where
 ${u} = (v-v_\star)/|v-v_\star|$ is the normalized relative
 velocity). In this work we are concerned with {\em short-range
   interaction} models.  More precisely we assume that $B$ is locally
 integrable. Here are the hypothesis on the collision kernel:
\begin{equation}
\label{HSkernel}
    B(|u|, \cos \theta) = \Phi(|u|)\, b(\cos\theta),
    \end{equation}
with 
\begin{equation} \label{hyp:phi} 
\Phi(z) = z^\gamma, \quad z \in \R_+, \quad \mbox{ for some } 
\gamma \in (0,+\infty)
\end{equation}
and $b$ smooth such that
\begin{equation}
\label{hyp:theta}
\int_0^\pi b(\cos\theta) \sin^{d-2}\theta \,d\theta < +\infty.
\end{equation}
These assumptions are satisfied for the so-called {\em hard spheres
  model} $B=|v-v_*|$, and it is known as {\em Grad's angular cutoff
  assumption} when it is (artificially) extended to interactions
deriving from a power-law potentials. As an important benchmark model
for the numerical simulation we therefore introduce the so-called {\em
  variable hard spheres model} (VHS), which writes
    \begin{equation*}
%\label{VHSkernel}
    B(|u|, \cos \theta) = C_\gamma \, |u|^\gamma,
    \end{equation*}
for some $\gamma \in (0,1]$ and a constant $C_\gamma >0$.
%% CM : à voir si on a besoin de garder la dernière phrase vu que 
%% dans ce papier on fera pas de simulations ? 

For this class of models, one can split the collision operator as
$$
Q(f,f)\,=\,Q^+(f,f)\,-\,L(f)\,f,
$$
with
\begin{equation*}
Q^{+}(f,f) \,=\, \int_{\R^d} \int_{\ens{S}^{d-1}} B(\vert v-v_\star \vert, \cos
\theta) f^\prime f^\prime _\star \,d\sigma \,dv_\star, 
%\label{eq:Qp}
\end{equation*}
and
\begin{equation*}
L(f) = \int_{\R^d} \int_{S^{d-1}} B(\vert v-v_\star \vert, \cos
\theta) f_\star \,d\sigma\,dv_\star. 
%\label{eq:Qm}
\end{equation*}

Boltzmann's collision operator has the fundamental properties of
conserving mass, momentum and energy: at the formal level
 \begin{equation*}
 \int_{{\R}^d}Q(f,f) \, \phi(v)\,dv = 0, \qquad
 \phi(v)=1,v,|v|^2,
 \end{equation*}
and it satisfies well-known Boltzmann's $H$ theorem
 \begin{equation*}
 - \frac{d}{dt} \int_{{\R}^d} f \log f \, dv = - \int_{{\R}^d} Q(f,f)\log(f) \, dv \geq 0.
 \end{equation*}
The functional $- \int f \log f$ is the {\em entropy} of the
solution. Boltzmann's $H$ theorem implies that any equilibrium
distribution function, {\em i.e.}, any function which is a maximum of the
entropy, has the form of a locally Maxwellian distribution
 \begin{equation*}
%\label{maxw}
 M(\rho,u,T)(v)=\frac{\rho}{(2\pi T)^{d/2}}
 \exp \left( - \frac{\vert u - v \vert^2} {2T} \right), %\label{eq:MAX}
 \end{equation*}
where $\rho,\,u,\,T$ are the {\em density}, {\em macroscopic velocity}
and {\em temperature} of the gas, defined by
 \begin{equation*}
%\label{field}
 \rho = \int_{v\in{\R}^d}f(v)\,dv, \quad u =
 \frac{1}{\rho}\int_{v\in{\R}^d}v\,f(v)\,dv, \quad T = {1\over{d\rho}}
 \int_{v\in{\R}^d}\vert u - v \vert^2\,f(v)\,dv.
 \end{equation*}
For further details on the physical background and derivation of
the Boltzmann equation we refer to Cercignani, Illner, Pulvirenti~\cite{Cerc}
and Villani~\cite{Vill:hand}.

\section{Formulation of a general stability result} \label{sec3}
\setcounter{equation}{0}

In this section we remind the basic principles leading to the
periodized truncations of the Boltzmann collision operator arising in
spectral methods. Then, we present the main result of this paper: the
stability of the spatially homogeneous Boltzmann equation with respect
to a smooth balanced perturbation, preserving mass and smoothness but
not nonnegativity of the solution.  This stability means that we are
able to construct global solutions and estimate the error between
perturbed and unperturbed solutions.
%% A voir ici : on peut tronquer en utilisant la représentation
%% usuelle ce n'est pas le point clef...
 
Any deterministic numerical method requires to work on a {\it bounded}
velocity space.  This therefore supposes a non physical truncation
(associated with limit conditions) of this velocity space, which we
shall discuss below. 

\subsection{General framework}
We consider the spatially homogeneous Boltzmann equation written in
the following general form
\begin{equation}
\frac{\partial f}{\partial t} = Q(f,f),
\label{eq:HOM}
\end{equation}
where $Q(f,f)$ is given by
\begin{equation}
Q(f,f) = \int_{\mathcal{C}} \mathcal{B}(y,z)\, \big[
f^\prime  f_\star^\prime - f_\star f \big] \,dy \,dz,\quad v\in\mathbb{R}^d
\label{eq:Qgen}
\end{equation}
with
 \begin{equation*}
 v^\prime   = v + \Theta^\prime(y,z), \qquad
 v^\prime_\star = v + \Theta^\prime_\star(y,z),  \qquad  v_\star = v +
 \Theta_\star(y,z).
 \end{equation*}

In the equations above, $\mathcal{C}$ is some given (unbounded)
domain for $y,z$, and $\Theta$, $\Theta^\prime$, $\Theta^\prime _\star$ are suitable
functions, to be defined later. This general framework emphasizes
the translation invariance property of the collision operator,
which is crucial for the spectral methods. We will be more precise in the
next paragraphs for some changes of variables allowing to reduce
the classical operator (\ref{eq:Q}) to the form (\ref{eq:Qgen}).

A problem associated with deterministic methods which use a fixed
discretization in the velocity domain is that the velocity space is
approximated by a finite region. Physically the domain for the
velocity is $\R^d$, and the property of having compact support is not
preserved by the collision operator. In general the collision process
indeed spreads the support by a factor $\sqrt{2}$ in the elastic case
(see \cite{a44,M:04} and also \cite{MM:04} for similar properties in
the inelastic case). As a consequence, for the continuous equation in
time, the function $f$ is immediately positive in the whole domain
$\R^d$. Thus, at the numerical level, some non physical condition has
to be imposed to keep the support uniformly bounded. In order to do
this there are two main strategies:

\begin{itemize}
\item One can remove the physical binary collisions that will lead
  outside the bounded velocity domains. This means a possible increase
  of the number of {\em local invariants}, {\em i.e.}, the functions
  $\varphi$ such that
  $$
  (\varphi^\prime_\star  \,+\, \varphi^\prime \,-\,\varphi_\star \,-\, \varphi)
  $$ 
  is zero everywhere on the domain. If this is done properly ({\it
    i.e.}, without removing too many collisions), the scheme remains
  conservative and without spurious invariants.  However, this
  truncation breaks down the convolution-like structure of the
  collision operator, which requires the translation invariance in
  velocity. Indeed the modified collision kernel depends on $v$
  through the boundary conditions. This truncation is the starting
  point of most schemes based on Discrete Velocity Models. 

\item One can add some non physical binary collisions by {\em
    periodizing} the function and the collision operator. This implies
  the loss of some local invariants (some non physical collisions are
  added). Thus the scheme is not conservative anymore, although it
  still preserves the mass if the periodization is done
  carefully. However in this way the structural properties of the
  collision operator are maintained and thus they can be exploited to
  derive fast algorithms. This periodization is the basis of spectral
  methods.
\end{itemize}

Therefore, we consider the space homogeneous Boltzmann equation in a
bounded domain in velocity $\mathcal{D}_L = [-L,L]^d$
($0<L<\infty$). We truncate the integration in $y$ and $z$ in
\eqref{eq:Qgen} since periodization would yield infinite result if
not: we set $y$ and $z$ to belong to some truncated domain
$\mathcal{C}_R \subset \mathcal{C}$ (the parameter $R$ refers to its
size and will be defined later).
% $x$ and $y$ such that $v_\star$, $v_\star^\prime$ and $v^\prime$
% vary in $B_R$, the ball of center $O$ and radius $R$.  For a
% compactly supported function $f$ with support $B_S$, we take $R$=$S$
% in order to obtain all possible collisions. Then, using a
% geometrical argument, we choose $T$, large enough with respect to
% $R$ (see below) in order to prevent intersections of the regions
% where $f$ is different from zero. Thus, the operator reads
For a compactly supported function with support included in $B_S$, the
ball centered at $0$ with radius $S>0$, one has to prescribe suitable
relations (depending on the precise change of variable and truncation
chosen) between $S$, $R$ and $L$ in order to retain all possible
collisions and at the same time prevent intersections of the regions
where $f$ is different from zero (this is the so-called {\em
  dealiasing condition}). Then the {\em truncated} collision operator
reads
\begin{equation}\label{eq:HOMtruncat}
Q^R(f,f) = \int_{\mathcal{C}_R} \mathcal{B}(y,z) \, 
\big( f^\prime_\star\,f^\prime \,-\, f_\star\,f \big) \,dy \, dz
\end{equation}
for $v \in \mathcal{D}_L$ (the expression for $v \in\R^d$ is deduced
by periodization).
% The interest of this representation is to preserve the real
% collision kernel and its properties.
By making some changes of variable on $v$, one can easily prove for
the two choices of variables $y,z$ of the next subsections, that for
any function $\varphi$ periodic on $\mathcal{D}_L$ the following weak
form is satisfied:
\begin{equation} \label{eq:QRweak}
\int_{\mathcal{D}_L} Q^R(f,f) \,\varphi(v) \,dv =
\frac{1}{4}\int_{\mathcal{D}_L}\int_{\mathcal{C}_R}
\mathcal{B}(y,z) \,f_\star\, f \, 
\left( \varphi^\prime_\star + \varphi^\prime - \varphi_\star - \varphi  \right)\, dy \,dz \, dv.
\end{equation}

Now, we use the representation $Q^R$ to derive spectral methods.
Hereafter, we use just one index to denote the $d$-dimensional sums
with respect to the vector $k=(k_1,..,k_d)\in \Z^d$, hence we set
$$
\sum_{k=-N}^N := \sum_{k_1,\dots,k_d=-N}^N.
$$
The approximate function $f_N$ is represented as the truncated Fourier
series
\begin{equation}
f_N(v) = \sum_{k=-N}^N \hat{f}_k \, e^{i \frac{\pi}L k \cdot v},
\label{eq:FU}
\end{equation}
with
$$
\hat{f}_k = \frac{1}{(2 L)^d}\int_{\mathcal{D}_L} f(v) \, 
e^{-i \frac{\pi}L k \cdot v }\,dv.
$$
In a Fourier-Galerkin method the fundamental unknowns are the
coefficients $\hat{f}_k(t)$, $\,k=-N,\ldots,N$. We obtain a set of
ODEs for the coefficients $\hat{f}_k$ by requiring that the residual
of (\ref{eq:HOMtruncat}) be orthogonal to all trigonometric
polynomials of degree less than $N$.  Hence for $k=-N,\ldots,N$
\begin{equation}
\int_{\mathcal{D}_L}
\left(\frac{\partial f_N}{\partial t} - Q^R(f_N,f_N)
\right)
e^{-i \frac{\pi}L k \cdot v}\,dv = 0.
\label{eq:VAR}
\end{equation}
By substituting expression (\ref{eq:FU}) in (\ref{eq:QRweak}) we get
$$
Q^{R}(f_N,f_N) = Q^{R,+}(f_N,f_N) -L^R(f_N)\, f_N
$$
with
\begin{eqnarray}
\label{qm}
L^R(f_N)\, f_N &=& \sum_{l=-N}^N\,\sum_{m=-N}^N \beta (m,m)\,
\hat{f}_l\,\hat{f}_m  \, e^{i \frac{\pi}L (l+m)
  \cdot v},
\\
\label{qp}
Q^{R,+}(f_N,f_N) &=& \sum_{l=-N}^N\,\sum_{m=-N}^N \beta (l,m)\, 
\hat{f}_l\,\hat{f}_m \, e^{i \frac{\pi}L
(l+m) \cdot v},
\end{eqnarray}
where %the so-called {\em kernel modes} $\beta$ are defined by
\begin{equation}
\label{beta}
\beta (l,m) = \int_{\mathcal{C}_R}
\mathcal{B}(y,z) e^{i \frac{\pi}L \big(l\cdot \Theta^\prime(y,z) + m\cdot \Theta_\star^\prime(y,z)\big)}  \,dy \,dz.
\end{equation}

The {\em spectral equation} is the projection of the collision
equation in $\PP_N$, the $(2N + 1)^d$-dimensional vector space of
trigonometric polynomials of degree at most $N$ in each direction,
{\it i.e.},
$$
\frac{\partial f_N}{\partial t} =\mathcal{P}_N\,Q^R(f_N,f_N),
$$
where $\mathcal{P}_N$ denotes the orthogonal projection on $\PP_N$ in
$L^2(\mathcal{D}_L)$.
% A straightforward computation leads to the following set of ordinary
% differential equations on the Fourier coefficients
%\begin{equation}
%\frac{\partial \hat{f}_k}{\partial t}
%= \sum_{{l+m=k}\atop{l,m=-N}}^N \hat{\beta} (l,m)\,\hat{f}_{l}\,\hat{f}_m,
%\label{eq:CF2}
%\end{equation}
%where $\hat{\beta} (l,m)$ are the so-called {\em kernel modes}, given by
%$$
%\hat{\beta} (l,m)= \beta (l,m) - \beta (m,m),
%$$
%with the initial condition
%\begin{equation}
%\hat{f}_k(0) = \frac{1}{(2L)^d}\int_{\mathcal{D}_L} f_0(v) \, 
%e^{-i \frac{\pi}L k \cdot v }\,dv.
%\end{equation}

%%%%%%%%%%%%%%%%%%%%%%%%%%%%%%%%%%%%%%%%%%%%%%%%%%%%%%%%%%%%%%%%%%%%%%%%%%%%%%%%%
\subsection{The truncation associated with classical spectral methods}

In the classical spectral method \cite{PaRu:spec:00}, a simple change
of variables in (\ref{eq:Q}) permits to write
\begin{equation}
Q(f,f) = \int_{\R^d} \int_{S^{d-1}}
\mathcal{B}^{class} (g, \omega)\big( f (v^\prime) f (v_\star^\prime)- f (v) f(v_\star) \big)  \,d\omega\,dg,
\label{eq:G}
\end{equation}
with $g = v - v_\star \in \R^d$, $\omega \in \ens{S}^{d-1}$, and
\begin{equation}
\left\{
\begin{array}{l}
v^{\prime} = v - \frac{1}{2}(g-\vert g\vert \omega ), \vspace{0.3cm} \\
v_\star^{\prime} = v - \frac{1}{2}(g+\vert g\vert \omega), \vspace{0.3cm} \\
v_\star = v +g.
\end{array}
\right.
\label{eq:VV2}
\end{equation}
Then, we set $ \mathcal{C} := \R^d \times \ens{S}^{d-1}$ and
$$
\Theta^\prime(g,\omega) :=  - \frac{1}{2}(g-\vert g\vert \omega ), \quad  
\Theta^\prime _\star(g,\omega) := - \frac{1}{2}(g+\vert g\vert \omega), \quad
\Theta_\star(g,\omega) := g.
$$
Finally the collision kernel $\mathcal{B}^{class}$ is defined by
\begin{equation}\label{eq:defBclassic}
\mathcal{B}^{class} (g,\omega) = 2^{d-1} \, \big( 1-(\hat{g} \cdot \omega) \big)^{d/2 -1}
B\big(|g|,2 (\hat{g}\cdot \omega)^2 -1 \big).
\end{equation}

Thus, the Boltzmann operator (\ref{eq:G}) is now written in the
form~(\ref{eq:Qgen}). 
% Moreover, from the conservation of
% momentum %$v^\prime_\star + v^\prime = v_\star + v$
% and energy %$|v_\star^\prime|^2 +|v^\prime|^2 = |v_\star|^2 +|v|^2$,
% we get the following result \cite{PePa:96}, assuming $\supp f \subset
% B_S$,
% \begin{itemize}
% \item we have $\supp Q(f,f) \subset B_{\sqrt{2}\,S}$,
% \item the collision operator is then given by
% $$
% Q(f,f)(v) = \int_{B_{2S}} \int_{S^{d-1}}
% B(|g|, \cos\theta) \big( f (v^\prime) f (v_\star^\prime) -
%  f (v_\star) f(v) \big) \, d\omega \,dg,
% $$
% with $v^\prime,v_\star^\prime,v_\star \in B_{ (2 + \sqrt{2})R}$.
% \end{itemize}
% As a consequence of this result, in order to write a spectral
% approximation to (\ref{eq:HOM}) we consider the distribution function
% $f$ restricted on $[-L,L]^d$,
% ($0<L<+\infty$), %(with $L \geq ( 2 + \sqrt{2} )S)$),
% assuming $f(v)=0$ on $[-L,L]^d \setminus B_S$, and extend it by
% periodicity to a periodic function on $[-L,L]^d$. We truncate the
% domain as $\mathcal{C} _R = B_R \times \ens{S}^{d-1}$ for $R>0$
% (defining thus $Q^R$). Following the previous discussion on the
% dealiasing condition, we take $R=2S$ and the shortest period can be
% restricted to $[-L,L]^d$, with $L\geq (3+{\sqrt 2})S/2$ (see for a
% more detailed discussion \cite{PaRu:spec:00}).
We consider the bounded domain $\mathcal{D}_L = [-L,L ]^d$, ($0<L
<\infty$) for the distribution $f$, and the bounded domain
$\mathcal{C}_R = B_R \times \ens{S}^{d-1}$ for some $R>0$.  The
truncated operator reads in this case
\begin{equation}
\label{eq:QRclass}
Q^R(f,f)(v)=\int_{B_R\times \ens{S}^{d-1}} \mathcal{B}^{class}
(g,\omega) \,
\big( f(v'_*) f(v') - f(v_*) f(v) \big) \, d\omega \, dg.
\end{equation}

% Then, we apply the spectral algorithm (\ref{qm}) and (\ref{qp}) and
% get the following {\em kernel modes} $\beta^c (l,m)$
%\begin{equation}
%\beta^c (l,m) = \int_{B_R} \int_{S^{d-1}}
%B(|g|, \cos\theta) \, e^{-i \frac{\pi}L \big( g\cdot\frac{(l+m)}{2} - i \vert g
%\vert\omega \cdot \frac{(m-l)}{2}\big)} \,d\omega\,dg.
%\label{eq:KM}
%\end{equation}
%%with $\lambda=1/(3+\sqrt{2})$.
%We refer to \cite{PaRu:spec:00,FiRu:04} for the explicit computation
% of Fourier coefficients $\beta^c (l,m)$ in the VHS case where $B$ is
% given by \eqref{VHSkernel}.  Now, the evaluation of the right-hand
% side of (\ref{eq:CF2}) requires exactly $O(N^{2d})$ operations. We
% emphasize that the usual cost for a DVM method based on $N^d$
% parameters for $f$ in the velocity space is $O(N^{2d}M)$ where $M$
% is the numbers of angle discretizations.
%%The loss term on the right-hand side is a convolution sum and thus
%%transform methods allow this term to be evaluated only in $0(N^d\log\,N)$ operations.
%%Hence the most expensive part of the computation is represented by the gain term.

\subsection{The truncation associated with fast spectral methods}

Here we shall approximate the collision operator starting from a
representation which conserves more symmetries of the collision
operator when one truncates it in a bounded domain.  This
representation was used in \cite{BoRj:HS:97,ibraRj} to derive finite
differences schemes and it is close to the classical Carleman
representation (cf. \cite{carl}).  The basic identity we shall need is
(for $u \in \R^d$)
\begin{equation}
\label{form1}
\frac{1}{2}  \int_{S^{d-1}} F(|u|\sigma - u) \,d\sigma =
\frac{1}{|u|^{d-2}} \int_{\R^d} \delta(2 \, y\cdot u + |y|^2) \, F(y) \, dy.
\end{equation}
Using (\ref{form1}) the collision operator (\ref{eq:Q}) can be written
as
\begin{multline}
\label{eq:Qnew}
Q(f,f)(v) = 2^{d-1} \int_{x\in\R^d} \int_{y\in \R^d} \mathcal{B}^{fast} (y,
z) \,\delta(y \cdot z) \,
\\
%\nonumber
\big( f(v + z) f (v + y)- f (v + y + z) f(v)\big) \, dy \, dz,
\end{multline}
with
$$
\mathcal{B}^{fast} (y, z)= 2^{d-1} \, B\left(|y+z|,
  -\frac{y\cdot(y+z)}{|y|\,|y+z|} \right) \,|y+z|^{-(d-2)}.
$$
Thus, the collision operator is now written in the form (\ref{eq:Qgen})
with $\mathcal{C} := \R^d \times \R^d$, 
$$
\mathcal{B}(y,z)=\mathcal{B}^{fast} (y,z)\,\delta(y\cdot z),
$$
and 
$$
v^\prime_\star = v + \Theta^\prime_\star(y,z),\quad v^\prime = v +
\Theta^\prime(y,z),\quad v_\star=v + \Theta_\star(y,z).
$$
with
$$
\Theta^\prime_\star(y,z) := z,\quad \Theta^\prime(y,z):= y,\quad
\Theta_\star(y,z):= y+z.
$$

Now we consider the bounded domain $\mathcal{D}_L = [-L,L ]^d$, ($0<L
<\infty$) for the distribution $f$, and the bounded domain
$\mathcal{C}_R = B_R \times B_R$ for some $R>0$. The (truncated)
operator now reads
\begin{equation}
\label{eq:QRfast}
Q^R(f,f)(v)=\int_{B_R\times B_R} \mathcal{B}^{fast} (y,z)\,\delta(y \cdot z)\,
\big( f(v+z) f(v+y) - f(v+y+z) f(v) \big) \, dy \, dz,
\end{equation}
for $v \in \mathcal{D}_L$. This representation of the collision kernel
yields better decoupling properties between the arguments of the
operator and allows to lower significantly the computation cost of the
method by using the fast Fourier transform (see~\cite{MP:03,FMP}).

Let us make a crucial remark about the choice of $R$. When $f$ has
support included in $B_S$, $S>0$, it is usual (see
\cite{PaRu:spec:00,MP:03}) to search for the minimal period $L$ (in
order to minimize the computational cost) which prevents interactions
between different periods of $f$ during {\em one} collision process.
From now on, we shall always assume that we can take $L$ and $R$ large
enough such that, when needed, $R \ge \sqrt 2 \, L$. Hence all the
torus is covered (at least once) in the integration of the variables
$(g,\omega)$ or $(y,z)$.

%show that we can take $R=\sqrt{2}S$ and
%$L\geq (3+{\sqrt 2})S/2$ in order to get all collisions and prevent
%intersections of the regions where $f$ is different from zero.

% From now, we can apply the spectral algorithm (\ref{qm}) and
% (\ref{qp}) to this collision operator and the corresponding kernel
% modes are given by
%$$
%\beta^f (l,m) = \int_{y\in B_R}\int_{z\in B_R} \tilde{B}(y,z)\,
% \delta(y\cdot z) \, e^{i \frac{\pi}L \, \big( l\cdot y + m\cdot z
%   \big)} \, dy \, dz.
%$$
%In the sequel we shall focus on $\beta^f$, and one easily checks that
% $\beta^f(l,m)$ depends only on $|l|$, $|m|$ and $|l \cdot m|$.

%\begin{remark}
%  Note that the classical spectral method originates the following
%  form of the kernel modes in the $y,z$ notation
%$$
%\beta^c (l,m) = \int_{y\in B_R} \int_{z\in B_R} \mathcal{B}^f (y, z)
% \, \delta(y\cdot z) \, \chi_{|y+z| \leq R} \, e^{i \frac{\pi}L \big(
%   l\cdot y + m\cdot z \big)} \, dy \, dz.
%$$
%One can notice that the condition $|y + z|^2 = |y|^2 + |z|^2 \leq
% R^2$ couples the modulus of $y$ and $z$, such that the ball is not
% completely covered (for instance, if $y$ and $z$ are orthogonal both
% with modulus $R$, the condition is not satisfied, since $|y +z|
% =\sqrt{2R}$). This explains the better decoupling properties between
% the argument of the collision operator of this representation.
%\end{remark}
 
\subsection{A common abstract formulation for the stability of
  spectral methods}

From now on, $Q^R$ shall denote a periodized truncated collision
operator as in (\ref{eq:QRclass}) or (\ref{eq:QRfast}).  As we shall
see, using this formulations, both classical and fast spectral methods
fall into the following framework:
\begin{equation} \label{perturb}
\left\{
\begin{array}{l}
  \displaystyle{\frac{\partial f}{\partial t} = Q^R(f,f) + P_\var (f),  
   \quad v \in \mathcal{D}_L,} 
  \\
  \,
  \\
  f(0,v)  = f_{0,\var}(v), \quad v \in\mathcal{D}_L, 
\end{array}\right.
\end{equation}
where $P_\var$ is a ``smooth balanced perturbation'', which means that
it satisfies the following (balanced law)
\begin{equation}
\label{hyp1:P}
 \int_{\mathcal{D}_L} P_\var(f) \, dv = 0
\end{equation}
and preserves the smothness of the distribution function, {\it i.e.},
there exists constants $\C_0,\C_k>0$ such that
\begin{equation}
\label{hyp2:P}
\left\{
\begin{array}{l}
  \displaystyle{\| P_\var(f) \|_{L^1} \,\,\le\,\, \C_0 \, \|f \|_{L^1} \, \|f\|_{L^1}} 
  \\
  \,
  \\
  \displaystyle{\| P_\var(f) \|_{H^k_{per}} \,\,\le\,\, 
    \C_k \, \|f \|_{L^1} \, \|f\|_{H^k_{per}}}, \quad k \ge 0,
\end{array}\right.
\end{equation}
where $\| \cdot \|_{H^k_{per}}$ is the usual norm of the Sobolev space
of periodic functions $H^k_{per}(\mathcal{D}_L)$.

Moreover the perturbation is supposed to be small in the following
sense: there exists a function $\varphi(\var)$ such that for any $p
\ge 0$,
\begin{equation}
\label{hyp3:P}
\| P_\var(f) \|_{H^p_{per}}\,\, \le\,\, \varphi(\var),
\end{equation}
where $\varphi(\var)$ depends on $\|f\|_{H^{p+k}_{per}}$ for some
$k>0$, and goes to zero as $\var$ goes to zero.

Finally in order to prove global existence with uniform regularity
bounds, we shall require additionnal assumptions on the relation
between the equilibrium distributions of the perturbed and unperturbed
(periodized) Boltzmann equations, and about the stability of the
unperturbed equation (see the following statement).

Let us therefore write the unperturbed equation for reference: 
\begin{equation}
\label{unperturb}
\left\{
\begin{array}{l}
\displaystyle{\frac{\partial f}{\partial t} = Q^R(f,f),  \quad v \in \mathcal{D}_L,} 
\vspace{0.2cm} 
\\
f(0,v)  = f_0(v), \quad v \in\mathcal{D}_L.
\end{array}\right.
\end{equation}

% Assume that there is some increasing sequence of subvector spaces
% $(\mathcal{S}_\var)_{\var >0}$ of $L^2(\mathcal{D}_L)$ such that
%\begin{itemize}
%\item $\mathcal{S}_\var \subset \cap_{k \ge 0} H^k$
%\item $\cup_{\var >0}\mathcal{S}_i = L^2$ 
%\item the associated orthogonal projectors $(\mathcal{P}_\var)_{\var
%    >0}$ satisfy
%$$\| (\mbox{Id}- \mathcal{P}_\var) \|_{H^{k+\alpha}_{per} \to H^k_{per}} \to 0 $$
%with explicit rate for any $k \ge 0$ and $\alpha >0$.
%\end{itemize}
%Then as soon as $f_0 \in \mathcal{S}_\var$ we have on the time
% interval as long as the solution exists in
% $H^k_{per}(\mathcal{D}_L)$:
% \begin{equation}
%\label{hyp4:P}
%\quad \int_{\mathcal{D}_L}\partial^s _v P_\var (f(t)) \, \partial^s _v f(t) \, dv =0 
%\end{equation}
%for any multi-index $s \in \N^N$ such that $|s| \le k$. In words, the
% perturbation ensures that some orthogonality property can be
% propagated as time goes on, which means that the perturbation can be
% forgotten for {\it a priori} estimates in Sobolev spaces.

Let us state the general stability theorem:

\begin{theorem} \label{th:existence} Let us consider a perturbed
  Boltzmann equation (\ref{perturb}) %collision operator $Q^R$
  in the torus $\mathcal{D}_L$, where $Q^R$ is defined by
  (\ref{eq:QRclass}) or (\ref{eq:QRfast}), and for a sequence of
  smooth balanced perturbations $\big(P_\var = P_\var(f)\big)_{\var >
    0}$ which satisfy~\eqref{hyp1:P}-\eqref{hyp2:P}-\eqref{hyp3:P}.%-\eqref{hyp4:P}.

  Assume that the constant functions are equilibra of the pertubed
  equation \eqref{perturb} (as for equation \eqref{unperturb}) and
  that they are nonlinearly locally stable in any
  $H^k_{per}(\cal{D}_L)$ for equation \eqref{unperturb}.

  We assume that $f_0$ is a non-negative function, non zero
  everywhere, belonging to $H^k _{per}(\mathcal{D}_L)$ with $k\in\N$
  and $k > d/2$.  We consider a sequence of smooth balanced
  perturbations $f_{0,\var}$ of the initial datum for the perturbed
  problem \eqref{perturb} (non necessarily positive) such that
  $$
  \int_{\mathcal{D}_L} f_{0,\var} = \int_{\mathcal{D}_L} f_0 \quad
  \mbox{ and } \quad \| f_0 - f_{0,\var} \|_{H^k_{per}} \,\,\leq \,\, \psi(\var)
  $$
  with $\psi(\var)$ goes to zero when $\var$ goes to zero.
  
  Then, there exists $\var_0 >0$ depending only on the collision
  kernel $B$, the truncation $R$, the constants
  in~\eqref{hyp2:P}-\eqref{hyp3:P} for the perturbation, and the
  $L^1(\mathcal{D}_L)$ and $H^k_{per}(\mathcal{D}_L)$ norms on $f_0$,
  such that for any $\var \in (0,\var_0)$,
  \begin{itemize}
  \item[$(i)$] there exists a unique global smooth solution $f_\var$
    to \eqref{perturb};
  \item[$(ii)$] for any $p<k$, this solution belongs to
    $H^p_{per}(\mathcal{D}_L)$ for all times with uniform bounds as
    time goes to infinity;
  \item[$(iii)$] this solution remains ``essentially non-negative''
    uniformly in time, in the sense that there is $\eta(\var) > 0$
    (with $\eta(\var) \rightarrow 0$ as $\var$ goes to $0$) such that
    the non-positive part is $\eta(\var)$-small:
    \[ 
    \forall \, t \ge 0, \quad \| f_\var^-(t,\cdot) \|_{L^\infty} \le
    \eta(\var)
    \]
    where $f_\var^-$ denotes $|f_\var| \, {\bf 1}_{\{f_\var \le 0\}}$;
  \item[$(iv)$] this solution $f_\var$ converges in
    $H^p_{per}(\mathcal{D}_L)$ for any $p<k$, uniformly on any
    $[0,T]$, $T >0$, to the solution $f$ of the unperturbed equation
    \eqref{unperturb} when the parameter $\varepsilon$ goes to zero;
  \item[$(v)$] the solution $f_\var$ to \eqref{perturb} converges in
    $H^k_{per}(\mathcal{D}_L)$ as time goes to infinity to the
    constant equilibrium distribution in the torus prescribed by its
    mass, and it is ``asymptotically uniformly positive'', that
    is for $t$ larger than some fixed explicit time.
  \end{itemize} 
\end{theorem}

We split the proof into two main steps: first in Section~\ref{sec4} we
prove existence, uniqueness and smoothness of a solution on an
arbitrary bounded time interval (as the size of perturbation goes to
$0$). The main difficulty is to prove that non-negativity of the
distribution function is recovered in a certain sense. Then, in
Section~\ref{sec5} we study the asymtotic behavior and establish
global stability in time. The main issue is there to prove
regularizing properties of the gain operator $Q^{R,+}(f,f)$ and
entropy production estimates on $Q^R$. Finally in Section~\ref{sec6}
we apply the previous general results to spectral methods, and prove
their stability and convergence. 
 
\section{Proof of stability on an arbitrary bounded time interval} \label{sec4}
\setcounter{equation}{0}

In this section we first give some technical lemmas and next establish
a result showing existence and uniqueness of a smooth solution on an
arbitrary time interval to the perturbed equation~\eqref{perturb},
depending on an assumption of smallness on the size of the
perturbation. Then, in Lemma~\ref{lmm4.4} we prove the control of
negative values of $f(t)$. 
% We set $T>0$ a fixed length for the time interval.

\subsection{Preliminary results}
We start this section by a classical result of $L^p$ estimates on the
Boltzmann operator $Q^R(g,h)$ given by
$$
Q^R(g,h) := \int_{\mathcal{C}_R} \mathcal{B}(y,z) \, \big( g^\prime_\star\,h^\prime \,-\, g_\star\,h \big) \,dy \, dz.
$$
\begin{lemma}
\label{lmm4.0} 
Let the collision kernel $B$ satisfy the assumption
\eqref{HSkernel}-\eqref{hyp:phi}-\eqref{hyp:theta}.  Then, the
periodized Boltzmann operator $Q^R$ (defined by (\ref{eq:QRclass}) or
(\ref{eq:QRfast})) satisfies: for all $p\in[0,\infty]$ there exists a
constant $\C_p(R,B)>0$ such that
\begin{equation} \label{lpQ} 
  \| Q^{R}(g,h) \|_{L^{p}}, \ \| Q^{R}(h,g)
  \|_{L^{p}}\,\ \leq\,\, \C_p(R,B) \, \| g \|_{L^1} \, \| h \|_{L^p}.
\end{equation}
%and
%\begin{equation}
%\label{lpQ2}
%\,\, \leq\,\, \C_p(R,B) \, \| g  \|_{L^p} \, \| h \|_{L^1}.
%\end{equation}  
\end{lemma}
\begin{proof}
  The proof is exactly similar to the case of the usual Boltzmann
  collision opertor for a collision kernel bounded with compact
  support, see for instance~\cite{MV:04} for a recent proof.
\end{proof}
 
Now, we prove smoothness of the solution to the perturbed problem
(\ref{perturb}) on a fixed time interval under the assumption of an
{\it a priori} bound on the $L^1$ norm of the solution.

\begin{lemma} \label{lem:l1hk} Let us consider a collision kernel $B$
  which satisfies the assumptions
  \eqref{HSkernel}-\eqref{hyp:phi}-\eqref{hyp:theta} and a sequence of
  smooth balanced perturbations $\big(P_\var =
  P_\var(f^\var)\big)_{\var > 0}$ which
  satisfy~\eqref{hyp1:P}-\eqref{hyp2:P}-\eqref{hyp3:P}, and let $T>0$ be
  the length of the time interval.  Assume that $f_0 \in
  H^k(\mathcal{D}_L)$ for $k\in\N$ and that $f(t)$ is a (non
  necessarily positive) solution to~\eqref{perturb} with initial datum
  $f_0$, which satisfies the $L^1$-estimate
  \begin{equation}
    \label{L1f}
    \forall \, t \in [0,T], \quad \| f(t) \|_{L^1} \le M.
  \end{equation}
  Then, there exists a constant $\C_{k}(M)>0$, only depending on $M$,
  $R$, $T$ and $\|f_0\|_{H^k_{per}}$ such that
  \begin{equation}
    \label{fhk}
    \forall \, t \in [0,T], \quad  \| f(t) \|_{H^k_{per}} \,\le\,  \C_k(M).
  \end{equation}
  \label{lmm4.1}
\end{lemma}

\begin{proof}[Proof of Lemma~\ref{lem:l1hk}]
  We proceed by induction on $k \ge 0$. 
  % to prove that there exists a constant $\C_k(M) >0$, only depending
  % on $M$, $T$, $R$ and
  % $\|f_0\|_{H^k_{per}}$ such that (\ref{fhk}) holds.
  For the first stage $k=0$, we apply Lemma~\ref{lmm4.0} with $p=2$
  and $g=h=f$ and we use assumption~\eqref{hyp2:P} on the
  perturbation:
\begin{eqnarray*}
\frac{1}{2}\,\frac{d}{dt} \| f(t)\|_{L^2}^2 &\leq& \| Q^R(f,f) \,+\, P_\var(f) \|_{L^2} \,\| f(t) \|_{L^2}
\\
 &\leq& \left( \C_2(R,B)  \,+\, \C_0  \right)\, \| f(t) \|_{L^1} \,\| f(t) \|_{L^2}^2.
\end{eqnarray*}
Hypothesis~\eqref{L1f} provides a control on $\| f(t)\|_{L^1}$ and we
can apply Gronwall's lemma to get the result (\ref{fhk}) at stage
$k=0$.

Let us now assume that~\eqref{fhk} holds at stage $k \ge 0$ and let us
prove that it also holds at stage $(k+1)$.

Let us first recall a formula on the derivatives of $Q^R$: from the
bilinearity of $Q^R$ and the translation invariance property of the
periodized Boltzmann collision operator, one has
\begin{equation*}
%\label{dvQ}
\nabla_v Q^R(f,f)  = Q^R(\nabla_v f, f ) +  Q^R(f,\nabla_v f),
\end{equation*}
which yields a Leibniz formula at any order $s\in\N$:
\begin{eqnarray}
\nonumber
\big\|Q^R(f,f)\big\|_{H^{s}_{per}}^2 &=& \sum_{|\nu|\leq s} \big\|\partial^\nu Q^R(f,f) \big\|_{L^2}^2
\\
\label{form:1}
&\le& \C \, \sum_{|\nu|\leq s}\, \sum_{|\mu|\leq |\nu|} \left(\begin{array}{l} \nu \\ \mu \end{array}\right) 
\big\| Q^R(\partial^\mu f,\partial^{\nu-\mu} f) \big\|_{L^2}^2.
\end{eqnarray}
Now,  using (\ref{form:1}) with $s=k+1$ we have 
\begin{equation}
\label{Qhk1}
\big\|Q^R(f,f)\|_{H^{k+1}_{per}}^2 \, \le \, \C \, \big\|Q^R(f,f)\big\|_{H^{k}_{per}}^2 \,
+\, \C \, \sum_{|\nu|=k+1}\, \sum_{|\mu|\leq k+1} \left(\begin{array}{l} \nu \\ 
\mu \end{array}\right) \big\| Q^R(\partial^\mu f,\partial^{\nu-\mu} f) \big\|_{L^2}^2.
\end{equation}
From Lemma~\ref{lmm4.0} with $p=2$ and $g=\partial^\mu f$, $h=\partial^{\nu-\mu} f$ together with the 
hypothesis~\eqref{fhk}, we get
\begin{equation}
\label{eq1:hk1}
\big\|Q^R(f,f)\big\|_{H^{k}_{per}} \,\,\leq\,\,  \C(k,R)\, \|f \|_{H^k_{per}}\,\| f \|_{H^k_{per}} \,\,\leq\,\, \C_2(k,R,B)\,\C_k(M)^2.
\end{equation}
Then we split the last term of (\ref{Qhk1}) in two parts for $\mu \neq 0$ and $\mu=0$. For $\mu\neq 0$, we again apply  Lemma~\ref{lmm4.0}  with $p=2$ and $g=\partial^\mu f$, $h=\partial^{\nu-\mu} f$  and use the fact that both derivatives $|\mu|=\sum_{i=1}^d |\mu_i|\leq k $ and $|\nu-\mu| \leq k$:
\begin{equation}
\label{eq2:hk1}
\sum_{|\nu|=k+1}\, \sum_{{|\mu|\leq k+1}\atop{\mu\neq 0}} \left(\begin{array}{l} \nu \\ \mu \end{array}\right) \| Q^R(\partial^\mu f,\partial^{\nu-\mu} f) \|_{L^2}^2 \,\leq \, \C_3(k,R,B)\,\C_k(M)^2.
\end{equation}
Finally, for $\mu=0$ we apply  Lemma~\ref{lmm4.0}  with $p=2$ and $g=f$, $h=\partial^{\nu} f$:
\begin{equation}
\label{eq3:hk1}
\sum_{|\nu|=k+1}\,  \| Q^R(f,\partial^{\nu} f) \|_{L^2}^2 \,\leq \,  \C_4(k,R,B)\,\|f\|_{L^1}\,\|f\|_{H^{k+1}_{per}}.
\end{equation}
Then, gathering inequalities from (\ref{eq1:hk1}) to (\ref{eq3:hk1}) and using the assumption (\ref{hyp2:P}) on the smooth balanced perturbation $P_\var$, we have
\begin{eqnarray*}
\frac{1}{2}\, \frac{d}{dt}\|f\|_{H^{k+1}_{per}}^2 &\leq& \|Q^R(f,f)\,+\,P_\var(f) \|_{H^{k+1}_{per}} \,\|f\|_{H^{k+1}_{per}} 
\\
&\leq& \C \, \left( \C_2 + \C_3 \right) \,\C_k(M)^2 \,\|f \|_{H^{k+1}_{per}} \,\,+\,\, \C \, \left( \C_4 \,+\,\C\right) 
\,\|f\|_{L^1} \,\|f \|_{H^{k+1}_{per}}^2.
\end{eqnarray*}
Finally using the control~\eqref{L1f} on $\|f(t)\|_{L^1}$ we apply Gronwall's lemma to get~\eqref{fhk} at stage $k+1$\,: 
there exists a constant $\C_{k+1}(M)$, only depending on $M$, $R$, $T$ and $\|f_0\|_{H^{k+1}_{per}}$ such that  
$$
\forall \, t \in [0,T], \quad \| f(t) \|_{H^{k+1}_{per}}\,\, \le\,\, \C_{k+1}(M).
$$
\end{proof}

Then, we establish existence and uniqueness of a smooth solution for the perturbed problem \eqref{perturb} 
on a small time interval $[0,\bar\tau]$, $\bar \tau >0$.
\begin{proposition}
\label{prop1} 
Let us consider a collision kernel $B$ which satisfies the assumptions 
\eqref{HSkernel}-\eqref{hyp:phi}-\eqref{hyp:theta} and a sequence of smooth balanced perturbations 
$\big(P_\var = P_\var(f)\big)_{\var > 0}$ which satisfy~\eqref{hyp1:P}-\eqref{hyp2:P}-\eqref{hyp3:P}. 
We assume that $f_0 \in H^k _{per}(\mathcal{D}_L)$, for $k\in\N$ and set 
\begin{equation}
\label{def:M}
M = 2\,\| f_0 \|_{L^1}. 
\end{equation}
Then, there exists $\bar\tau = \bar \tau(M) >0$ such that for all
$\var >0$ the perturbed Boltzmann equation \eqref{perturb} admits a
unique solution (non necessarily positive) on the time interval
$[0,\bar\tau]$, where the solution $f$ satisfies
\begin{equation}
\label{res1:prop1}
\forall \, t \in [0,\bar\tau], \quad \| f(t) \|_{L^1} \,\,\le\,\, M.
\end{equation}
Moreover, there exists a constant $\C_k(M)>0$, only depending on $M$,
$R$, and $\|f_0\|_{H^k_{per}}$ such that
\begin{equation}
\label{res2:prop1}
\forall \, t \in [0,\bar\tau], \quad \| f(t) \|_{H^k_{per}} \,\,\le\,\, \C_k(M).
\end{equation}
\end{proposition}
\begin{proof}[Proof of Proposition~\ref{prop1}]
First, we apply Lemma~\ref{lmm4.0} with $p=1$ and $g=h=f$
\begin{eqnarray*}
\left\| Q^R(f,f) \right\|_{L^1} \leq \C_1(R,B) \,\|f \|_{L^1}\, \|f \|_{L^1}.
\end{eqnarray*}
Moreover, using assumption \eqref{hyp2:P} on the perturbation, there
exists $\C>0$ such that for all $\var>0$
$$
\| P_\var(f) \|_{L^1} \le \C \, \| f \|_{L^1}^2.
$$
Therefore, we obtain a constant $\C>0$, only depending on $R$ and the
collision kernel $B$ such that
$$
\frac{d}{dt} \| f(t) \|_{L^1} \leq  \C \,\|f\|_{L^1}^2.
$$
This implies that
$$
 \| f(t) \|_{L^1} \,\,\leq\,\,  \frac{ \| f_0 \|_{L^1}}{1\,-\,\C\,\|f_0\|_{L^1} \,t}.
$$
Now, setting $M=2\,\|f_0\|_{L^1}$ and from the latter inequality, we
show that there exists $\bar\tau \,<\, 1/(2\,C\,\|f_0\|_{L^1})$ such
that
\begin{equation*}
 \forall\, t\,\in\,[0,\bar\tau], \quad \| f(t) \|_{L^1} \leq M,
\end{equation*}
which gives (\ref{res1:prop1}).

From the estimate in $L^1(\mathcal{D}_L)$ on the function $f(t)$ on
the time interval $[0,\bar\tau]$, we prove existence and uniqueness of
a solution by Cauchy-Lipschitz theorem in $L^1(\mathcal{D}_L)$
(because of the truncation on $\mathcal{D}_L$ the collision kernel is
a bounded bilinear function from $L^1(\mathcal{D}_L) \times
L^1(\mathcal{D}_L)$ to $L^1(\mathcal{D}_L)$).  Finally, from the bound
\eqref{res1:prop1} and the smoothness assumption $f_0\in
H^k_{per}(\mathcal{D}_L)$ on the initial datum, we are able to apply
Lemma \ref{lmm4.1}, which proves that there exists a constant
$\C_{k}(M)>0$, only depending on $M$, $R$, $T$ and
$\|f_0\|_{H^k_{per}}$ such that
\begin{equation*}
 \| f(t) \|_{H^k _{per}} \,\le\,  \C_k(M).
\end{equation*}
This concludes the proof.
\end{proof}

By iterating Proposition \ref{prop1} and Lemma \ref{lmm4.1}, we
observe that uniform control on the $L^1(\mathcal{D}_L)$ norm on an
arbitrarily large time interval $[0,T]$ together with smoothness on
the initial datum will ensure existence and uniqueness of a smooth
solution on this time interval $[0,T]$. Furthermore we observe that
the control on the $L^1(\mathcal{D}_L)$ norm is obvious for the
classical Boltzmann equation thanks to the positivity and mass
preservations.  Therefore, we shall now focus on the control of
positivity of the solution, showing that the solution remains ``almost
positive'' for arbitrarily large time interval, hence allowing to
produce uniform control on the $L^1(\mathcal{D}_L)$ norm for
arbitrarily large time interval.
 
We first state a technical lemma which takes advantage of the mixing
property of $Q^{R,+}$ in order to show spreading of the support of a
characteristic function of a ball.
\begin{lemma} \label{lem:spreadQ} Let us consider a collision kernel
  $B$ which satisfies the assumptions
  \eqref{HSkernel}-\eqref{hyp:phi}-\eqref{hyp:theta} and a truncated
  operator $Q^R$ defined by (\ref{eq:QRclass}) or (\ref{eq:QRfast}).
  Then for all $0<r<\sqrt 2 \, L$, we have
  $$
  Q^{R,+} ({\bf 1}_{B(v,r)}, {\bf 1}_{B(v,r)}) \ge C_0 \, {\bf
    1}_{B(v,\mu \, r)}
  $$
  for some explicit $\mu = \mu(R,L) >1$ and $C_0$.
%  $$
%  0 < \mu < \mu_0 := \left( 1 + \frac{R}r \, \left( 2 - \left( \frac{R}r
%      \right)^2 \right)^{1/2} \right)^{1/2} > 1
%  $$
%  (where the constant $C(\mu, r)$ may blow-up as $\mu \to \mu_0$).
\end{lemma}

\Remark Note that for $r \ge \sqrt 2 \, L$, one has ${\bf 1}_{B(v,r)} =
1$ on the torus $[-L,L]^d$ and there is nothing to prove.
%$\mu_0>1$ depends continuously on $r,R$ and
%therefore for $r \in [0,L]$ (periodic bounded domain) and $R \le 2 \,
%r$ we can find
%$$
%1 < \bar \mu < \min_{0 \le r \le L, R \le 2 \, r}\mu_0(r,R).
%$$
%In the sequel we shall always use such a uniform spreading parameter
%$\bar \mu$.  
\medskip

\begin{proof}[Proof of Lemma~\ref{lem:spreadQ}]
  The invariance by translation allows to reduce the proof to the case
  $v = 0$. The invariance by rotations implies that $I_r :=
  Q^{R,+}({\bf 1}_{B(0,r)}, {\bf 1}_{B(0,r)})$ is radially
  symmetric. More precisely, taking a $C^\infty$ radially symmetric
  function $\phi$ such that $\phi > 0$ on $B = B(0,r)$ and $\phi \le
  {\bf 1}_B$ on $\R^d$, we have
  \begin{itemize}
  \item the function $v \,\longrightarrow\, Q^{R,+}(\phi,\phi)(v)$ is
    continuous,
  \item for all $v\in\R^d$, $I_r(v) \ge Q^+(\phi,\phi)(v)$,
  \item for all $v\in B$, $Q^{R,+}(\phi,\phi)(v) > 0$. 
  \end{itemize}
  As a consequence, for any ball $B'=B(0,r^\prime)$ strictly included
  in $B$, there exists $\kappa_{r^\prime} > 0$ such that $I_r \ge
  \kappa_{r^\prime} \, {\bf 1}_{B'}$.

  In order to conclude, we just need to estimate the support of $I_r$
  close to the ball $B$.

  Let us fix $r^\prime \in (0,r)$ and choose $v^\prime, v^\prime
  _\star \in B(0,r^\prime)$ such that $|v^\prime| = |v^\prime _\star|
  = r^\prime$ and $|v^\prime-v^\prime_\star|= \min\{ R ; \sqrt 2
  r^\prime\}$.  Then taking $\sigma$ in $B(0,r^\prime)$ such that
  $v,v_\star,v^\prime,v^\prime_\star$ is a square, we find 
  $|v| = \sqrt 2 \, r'$ if $\sqrt 2 \, r' \le R$ and 
  $$
  |v| = \left[ (r^\prime)^2 - \frac{R^2}4 \right]^{1/2} + \frac{R}{2}
  $$
  else. We define
  $$
  \mu_0 = \min_{R/\sqrt 2 \le r' \le \sqrt 2 \, L} 
  \left[ 1 - \frac{R^2}{4 (r^\prime)^2} \right]^{1/2} + \frac{R}{2
    r^\prime} = \min_{R/(2 \sqrt 2 L) \le y \le 1/\sqrt 2}
  \left( \sqrt{1-y^2} + y \right) \in (1,\sqrt 2).
  $$
  % $$
%   |v| = \left( (r^\prime)^2 + R \, \sqrt{2 \, (r^\prime)^2 - R^2}
%   \right)^{1/2} = \left( 1 + \frac{R}{r^\prime} \, \sqrt{2 -
%       \frac{R^2}{(r^\prime)^2}} \right)^{1/2} \, r^\prime =:
%   \mu_0(r^\prime,R) \, r^\prime
%   $$
%   with $\mu(r^\prime,R) >1$. 
  This concludes the proof: we deduce that
  for any $v$ such that $|v| \in (r^\prime,\mu_0 \,
  r^\prime)$, we have
  $$
  I_r (v) \ge Q^+(\phi,\phi)(v) >0 
  $$
  since $\phi$ is strictly positive in the neighborhood of the
  $v^\prime,v^\prime_\star$ associated to $v$ constructed above.
  Hence we deduce by taking $r^\prime<r$ close to $r$ that for any $0
  < \mu < \mu_0$, we have
  $$
  I_r \ge C(\mu,r) \, {\bf 1}_{B(0,\mu \, r)}
  $$
  for some constant $C(\mu,r) >0$ depending continuously on $r$. 
  We can choose $\mu = (1+\mu_0)/2$ for instance, and, for this choice
  of $\mu$, we take 
  $$
  C_0 = \min_{0 \le r \le \sqrt 2} C(\mu,r) >0.
  $$
\end{proof}

Finally we establish the following positivity result on the solution
to the perturbed problem (\ref{perturb}):

\begin{lemma}
\label{lmm4.4}
Let us consider a collision kernel $B$ which satisfies the assumptions
\eqref{HSkernel}-\eqref{hyp:phi}-\eqref{hyp:theta}, a truncated
operator $Q^R$ defined by (\ref{eq:QRclass}) or (\ref{eq:QRfast}), and
a sequence of smooth balanced perturbations $\big(P_\var =
P_\var(f)\big)_{\var > 0}$ which
satisfy~\eqref{hyp1:P}-\eqref{hyp2:P}-\eqref{hyp3:P}. 

We assume that $f_0$ is a non-negative function such that $f_0 \in H^k
_{per}(\mathcal{D}_L)$ with $k\in\N$ and $k > d/2$.
Moreover, we define $M=2\,\|f_0\|_{L^1}$ and $f_{0,\var}$ a smooth
balanced perturbation of $f_0$, which is non necessarily positive and
such that
$$
\int_{\mathcal{D}_L} f_{0,\var} = \int_{\mathcal{D}_L} f_0 \quad
  \mbox{ and } \quad \| f_0 - f_{0,\var} \|_{H^k_{per}} \,\,\leq \,\, \psi(\var)
$$
where $\psi(\var)$ goes to zero when $\var$ goes to zero.
We also set $\bar \tau, \C_k(M)>0$ the constant defined in
Lemma~\ref{lmm4.1} such that from Proposition~\ref{prop1} we have
$$ 
\forall \, t \in [0,\bar \tau], \quad \| f(t) \|_{H^k _{per}} \,\leq \,\C_k(M). 
$$

Then, there exists $\hat\tau \in (0,\bar \tau)$ which only depends on
$M$, $R$ and the collision kernel $B$, and there exists $\hat \var
>0$ which only depends on $\hat\tau$, $\C_k(M)$, such that for all
$\var$ such that $0<\var< \hat\var$ and for any smooth solution $H^k
_{per}(\mathcal{D}_L)$ to the perturbed Boltzmann equation
\eqref{perturb} with perturbed initial datum $f_{0,\var}$, we have
$$ 
\forall \, v \in \mathcal{D}_L, \quad f_\var (\hat\tau,v) > 0.
$$
Moreover there exists $\eta(\var)$, which goes to $0$ as $\var$ goes
to zero, such that the non-positive part of $f$ satisfies
\begin{equation}
\label{c::1}
\| f^-(t)\|_{L^\infty} \,\leq\, \eta(\var), \quad t\in[0,\hat\tau],  
\end{equation}
\end{lemma}
\begin{proof}[Proof of Lemma~\ref{lmm4.4}] 
  Let $\bar\tau >0$ be the length of the time interval for which there
  exists a smooth solution to the perturbed Boltzmann equation
  \eqref{perturb} with perturbed initial datum $f_{0,\var}$ such that
  (in the following we omit the subscript $\var$ for the solution)
  $$
  \| f(t)\|_{H^k_{per}} \,\leq \C_k(M), \quad t\,\in\,[0,\bar\tau].
  $$ 
  We split the proof into three steps: first, we give a classical
  estimate on the loss term $L^R(f)$, second we establish an estimate
  of $f^-$ with respect to the amplitude of the perturbation
  $P_\var(f)$, and third we use the spreading properties of
  $Q^{R,+}(f,f)$ to prove that there exists $\hat\tau \in (0,\bar
  \tau)$ such that
  $$
  f(\hat\tau,v) > 0, \qquad v \in \mathcal{D}_L. 
  $$ 
\smallskip

\noindent {\bf Step~1.} Applying Proposition~\ref{prop1} for the
control of $\|f(t)\|_{L^1}$ on the time interval $[0,\bar\tau]$, we
get
\begin{equation}
\label{est1}
\|L^R(f)\|_{L^\infty} \,\leq\, C(R,B) \,\|f\|_{L^1} \,\leq \, C(R,B) \,M,
\end{equation}
which gives  for all $ t\,\in\,[0,\tau_0]$, with $\tau_0=\min\left\{\bar\tau, \ln 2 /(M\, C(R,B))\right\}$
\begin{equation}
\label{est1-1}
2\,\geq e^{ M\, C(R,B) \,t} \ge \,e^{-\int_0^t L^R(f(s))\,ds} \,\geq\, e^{-\,M\, C(R,B) \,t} \,\geq\, \frac{1}{2}.
\end{equation}
\smallskip

\noindent {\bf Step~2.} Let us split $f$ as $f=f^+ - f^-$, with
$f^\pm= \max\{0,\pm f\}$ and use the monotonicity of $Q^{R,+}$ for
nonnegative distribution functions; it follows that
\begin{eqnarray}
\nonumber
Q^{R,+}\left(f,f\right) &=& Q^{R,+}\left(f^+ - f^-,f^+ - f^-\right)
\\
\label{c:0}
&\geq & -\,\left[  Q^{R,+}\left( f^+,f^- \right) +  Q^{R,+}\left(f^-,f^+\right)\right].
\end{eqnarray}
On the time interval $[0,\tau_0]$, we apply Proposition \ref{prop1} to
estimate $\|f(t)\|_{L^1}$ and since $\|f^+(t)\|_{L^1} \leq
\|f(t)\|_{L^1}$; we get from Lemma \ref{lmm4.0}
\begin{equation}
\label{Qpfpfm}
\| Q^{R,+}\left( f^+,f^- \right)\|_{L^\infty}, \,\, \| Q^{R,+}\left( f^-,f^+ \right)\|_{L^\infty} \,\,\leq\,\, \C_\infty(R,B)\, M\, \|f^-(t)\|_{L^\infty}, 
\end{equation}
which yields using (\ref{c:0})
\begin{equation}
\label{QplusLinf}
Q^{R,+}\left(f,f\right) \,\geq\, -2\, \C_\infty(R,B)\, M\, \|f^-(t)\|_{L^\infty}.
\end{equation}
Thus from the Duhamel representation of the solution $f$, we have for
$v\in\mathcal{D}_L$,
\begin{eqnarray*}
  f(t,v) &=& f_{0,\var}(v) \,e^{-\int_0^t L^R(f(s))(v)\,ds} 
  \\
  && + \int_0^t \left[ Q^{R,+}(f(s),f(s)) + P_\var(f(s))\right](v)\,e^{-\int_s^t L^R(f(u))(v)\,du}\,ds,
  \\
  &\geq & - 2 \, \psi(\var) - \int_0^t \left[ Q^{R,+}(f(s),f(s)) + P_\var(f(s))\right](v)\,e^{-\int_s^t L^R(f(u))(v)\,du}\,ds.
\end{eqnarray*}
Hence, we get from the lower estimate (\ref{QplusLinf}) of
$Q^{R,+}(f,f)$ and the smallness assumption (\ref{hyp3:P}) of the
perturbation $P_\var(f)$, for all $v\in\mathcal{D}_L$
$$
f^-(t,v) \,=\, \max\{0,-f(t,v)\} \,\leq\, 2 \, \psi(\var) + 
2\, \int_0^t \left( 2\, \C_\infty(R,B)\, M\, \|f^-(s)\|_{L^\infty} + \varphi(\var) \right) \,ds.
$$
Finally, we take the supremum in $v\in\mathcal{D}_L$ and apply
Gronwall's lemma to get for any $0\leq\tau\leq \tau_0$
\begin{equation*}
%\label{c::1}
\| f^-(t)\|_{L^\infty} \,\leq\, \big( 2 \, \psi(\var) 
+ 2 \, \tau_0 \, \varphi(\var) \big) \,
e^{4\, \C_\infty(R,B)\, M\,t}:=\eta(\var), \quad t\in[0,\tau_0],  
\end{equation*}
which proves (\ref{c::1}) on the time interval $[0,\tau_0]$.
\smallskip

\noindent {\bf Step~3.} Let us prove that there exists $\hat\tau \in
(0,\tau_0)$ such that
$$
f(\hat\tau,v) >0, \quad v\in\mathcal{D}_L. 
$$
We start again with the Duhamel representation of the solution
\begin{eqnarray*}
f(t,v) &=& f_{0,\var}(v) \,e^{-\int_0^t L^R(f(s))(v)\,ds} 
\\
&& + \int_0^t \left[ Q^{R,+}(f(s),f(s)) + P_\var(f(s))\right](v)\,e^{-\int_s^t L^R(f(u))(v)\,du}\,ds,
\end{eqnarray*}
but we now take into account the fact that the first term is
essentially positive and we use the spreading property of the operator
$Q^{R,+}$ (Lemma~\ref{lem:spreadQ}).

On the one hand since the initial datum is smooth enough ($k > d/2$ is
large enough such that $f_{0,\var}$ is H\"older), there exists an
explicit $\delta>0$ depending on $\C_k(M)$ such that for $\var$ small
enough, there exists $v_0\in\mathcal{D}_L$ such that
$$
f_{0,\var}(v) \,\,\geq\,\, \frac{\eta}{2} \,{\bf 1}_{B(v_0,\delta)}(v)
- \psi(\var), \quad{\rm with \,\,} 
\eta \,=\, \frac{\|f_0\|_{L^1}}{(2L)^d}.
$$

On the other hand, using the lower bound on the gain operator
(\ref{QplusLinf}) and the estimate (\ref{c::1}) of $f^-(t)$, it gives
$$
Q^{R,+}\left(f,f\right) \,\geq\, - 2\, \C_\infty(R,B)\,
M\,\eta(\var) 
%+ 2 \, \varphi(\var) \, \tau_0 \big) \,e^{2\, \C_\infty(R,B)\, M\,t}.
$$  
Finally, using (\ref{est1-1}) and the smallness assumption on
$P_\var(f)$, it first yields for any $0<\tau<\tau_0$
$$
f(t,v) \,\geq\,  A_0\,{\bf 1}_{B(v_0,\delta)}(v) \, 
- \, \varphi_1(\var),
$$
with 
\begin{equation}
\label{A0B0}
A_0 : = \frac{\eta}{4}, \qquad 
\varphi_0(\var) : =  2 \, \tau_0 \, 
\big( 2\, \C_\infty(R,B)\, M\, \eta(\var) + \varphi(\var) \big) \,+\, \psi(\var).
\end{equation}
Now we choose $\var_1>0$ small enough such that
$$
0\,<\, \frac{A_0}{2} \,\leq\, A_0 -  \varphi_0 (\var_1) \quad
\Longleftrightarrow \quad\varphi_0 (\var_1) \,\leq\,\frac{A_0}{2}. 
$$
Thus, we get on any  time intervall $[0,\tau]\subset[0,\tau_0]$
$$
f^+(t,v) \,\geq\, \frac{A_0}{2}\, {\bf 1}_{B(v_0,\delta)}(v).
$$
Hence, using the spreading properties of the operator $Q^{R,+}(f,f)$
of Lemma~\ref{lem:spreadQ} and the monotonicity of $Q^{R,+}$ for
nonnegative distribution functions, it follows that
\begin{eqnarray*}
Q^{R,+}(f^+,f^+) &\geq& \frac{A_0^2}{4}\,
Q^{R,+}\big({\bf 1}_{B(v_0,\delta)},{\bf 1}_{B(v_0,\delta)}\big) \\
&\geq& \frac{A_0^2}{4}\,C_0 \, {\bf 1}_{B(v_0, \mu \, \delta)}.
\end{eqnarray*}
Next, we again use the uniform bounds previously established in
\eqref{Qpfpfm} on $Q^{R,+}(f^+,f^-)$ and $Q^{R,+}(f^-,f^+)$:
\begin{eqnarray*}
Q^{R,+}(f,f) &\geq& Q^{R,+}(f^+,f^+) - Q^{R,+}(f^-,f^+) -Q^{R,+}(f^+,f^-) \\
&\geq& \frac{A_0^2}{4}\,C_0 \, {\bf 1}_{B(v_0, \mu \, \delta)} 
- 2\, \C_\infty(R,B)\, M\,\big( \eta(\var) + \varphi_0(\var) \big). 
\end{eqnarray*}
Finally, from the smallness assumption (\ref{hyp3:P}) of the
perturbation $P_\var(f)$, it yields for $t\in[\tau/2, \tau]$ and with
(\ref{est1-1})
\begin{eqnarray*}
f(t,v) &\geq & - \psi(\var) 
+ \int_0^t \left[ Q^{R,+}(f(s),f(s)) 
+ P_\var(f(s))\right](v)\,e^{-\int_s^t L^R(f(u))(v)\,du}\,ds,
\\
&\geq &  \frac{\tau}{2}\,\frac{A_0^2}{4}\,C_0 \, 
  {\bf 1}_{B(v_0,\mu \, \delta)} -\, 2 \, \tau\, \left[ 2\, \C_\infty(R,B)\,
  M\, \big( \eta(\var) + \varphi_0 (\var) \big)\,+\, \varphi(\var) \right]
- \psi(\var) \\
& = & A_1 \,{\bf 1}_{B(v_0, \mu \, \delta)} \,-\, \varphi_1(\var)
\end{eqnarray*}
with 
$$
A_1  = \frac{\tau}{8}\,A_0^2\,C_0, \qquad 
\varphi_1(\var) := 2 \, \tau_0\, \left[ 2\, \C_\infty(R,B)\,
  M\, \big( \eta(\var) + \varphi_0 (\var) \big)\,+\, \varphi(\var) \right]
- \psi(\var) .
$$

Now, we proceed by induction: assume that there exists
$(A_j,\var_j,\varphi_j)$ such that on the time interval
$[\tau-\tau/2^{j},\tau] \subset [0,\tau_0]$ and for $\var \in (0,
\var_j)$, we have
$$
f(t,v) \,\,\geq \,\, A_{j}\, \,{\bf 1}_{B(v_0,
  \mu^{j}\,\delta)}(v) \,-\,
\varphi_{j}(\var),
$$
%where $\var_{j}$  was small enough such that
%$$
%0 \,<\, \frac{A_{j-1}}{2} \,\leq \, A_{j-1} \,- \, \varphi_{j-1}(\var)
 %\quad \Leftrightarrow\quad \varphi(\var_j) \,\leq\, \frac{A_{j-1}}{2\,B_0\,\tau}. 
%$$ 
where
$$
A_j := \left( \frac{\tau}{8}\right)^{2^j-1}\,A_0^{2^j} \,C_0 ^{2^{j-1}-1} 
$$
and $\varphi_j(\var) \to 0$ as $\var \to 0$. 

Using the same method as before, we first set $\var_{j+1}$ such that
$$
f^+(t,v) \,\geq\, \frac{A_j}{2} {\bf 1}_{B(v_0, \mu^{j}\,\delta)}
$$
and prove that for $t\in[\tau-\frac{\tau}{2^{j+1}},\tau]$ 
\begin{eqnarray*}
Q^{R,+}(f,f)  &\geq& \,  \left(\frac{A_j}{2}\right)^2 \,C_0 
\,{\bf 1}_{B(v_0, \mu^{j+1}\,\delta)} 
- \, 2\, \C_\infty(R,B)\, M\, \big( \eta(\var) + \varphi_{j}(\var) \big). 
\end{eqnarray*}
Then, from the Duhamel formula and the smallness assumption of
$P_\var(f)$, we finally get the following lower bound
$$
f(t,v)\,\, \geq\, \, A_{j+1} \,{\bf 1}_{B(v_0, \mu^{j+1}\,\delta)} 
\,-\, \varphi_{j+1}(\var),
$$ 
with 
$$
A_{j+1}  = \frac{\tau}{8}\,A_j ^2\,C_0, \qquad 
\varphi_{j+1}(\var) := 2 \, \tau_0\, \left[ 2\, \C_\infty(R,B)\,
  M\, \big( \eta(\var) + \varphi_{j} (\var) \big)\,+\, \varphi(\var) \right]
- \psi(\var) .
$$

Since $\mu>1$, the ball $B(v_0,\mu^{j}\,\delta)$ eventually
recovers the periodic box $[-L,L]^d$ {\it i.e.}, for some $J$ large
enough: $[-L,L]^d\subset B(v_0,\mu^{J}\,\delta)$, and for all $t\in
[\tau-\frac{\tau}{2^J},\tau]$, by applying $J$'s times the previous
induction we get for $\var \in (0,\var_{J})$:
$$
f(t,v) \,\,\geq \,\, A_{J} \,{\bf 1}_{B(v_0,\mu^{J}\,\delta)}(v)
\,-\, \varphi_{J}(\var).
$$

Finally, up to reducing $\var$ further, we have proved that there exists $(\hat\tau,\hat\var)$ 
which only depend on the collision kernel $B$, the initial datum $f_0$, $L$ and the perturbation function 
$\varphi= \varphi(\var)$ such that for all $0 < \var < \hat\var$,
$$
\forall \, v\in \mathcal{D}_L, \quad f(\hat\tau,v) > 0. 
$$
\end{proof}

% Note that from the conservation of mass, it yields that at time
% $\hat\tau$
%$$
%\|f(\hat\tau)\|_{L^1} = \int_{\mathcal{D}_L} f(\hat\tau,v) dv =
% \int_{\mathcal{D}_L} f(0,v) dv.
%$$ 
%This is the crucial property we shall exploit to prove
% Theorem~\ref{th:existence}.

\subsection{Existence and regularity on a bounded time interval}

\begin{proposition} 
\label{lem:bdedtime}

Let us consider a fixed time $T >0$, a collision kernel $B$ which
satisfies the assumptions
\eqref{HSkernel}-\eqref{hyp:phi}-\eqref{hyp:theta}, a truncation $Q^R$
defined by (\ref{eq:QRclass}) or (\ref{eq:QRfast}), and a sequence of
smooth balanced perturbations $\big(P_\var = P_\var(f)\big)_{\var >
  0}$ which satisfy~\eqref{hyp1:P}-\eqref{hyp2:P}-\eqref{hyp3:P}.

We assume that $f_0$ is a non-negative function, not zero everywhere,
and such that $f_0 \in H^k _{per}(\mathcal{D}_L)$ with $k\in\N$ and $k
> d/2$.  We define $M=2\,\|f_0\|_{L^1}$ and $(f_{0,\var})_{\var >0}$ a
sequence of smooth perturbations of $f_0$ (which is non
necessarily positive) such that
$$
\int_{\mathcal{D}_L} f_{0,\var} = \int_{\mathcal{D}_L} f_0 \quad
  \mbox{ and } \quad  \| f_0 - f_{0,\var} \|_{H^k_{per}} \,\,\leq \,\, \psi(\var)
$$
where $\psi(\var)$ goes to zero when $\var$ goes to zero.

Then, there exists $\hat \var >0$, which only depends on the
$H^k_{per}(\mathcal{D}_L)$ and $L^1(\mathcal{D}_L)$ norms of $f_0$,
such that for all $\var \in (0,\hat\var)$,
\begin{itemize}

\item[$(i)$] there is a unique smooth solution
  $f_\var=f_\var(t,\cdot)$ on $[0,T]$ to the perturbed equation
  \eqref{perturb} with initial datum $f_{0,\var}$;

\item[$(ii)$] this belongs to $H_{per}^k(\mathcal{D}_L)$ (with bound
  growing at most exponentially);

\item[$(iii)$] there is some explicit $\eta(\var) > 0$ (with
  $\eta(\var) \rightarrow 0$ as $\var$ goes to $0$) such that the
  non-positive part is $\eta(\var)$-small:
\[ 
\forall \, t \in [0,T], \quad \| f_\var^-(t,\cdot) \|_{L^\infty} \le
\eta(\var)
\]
where $f_\var^-$ denotes $\max \{ 0, - f_\var\}$; 

\item[$(iv)$] this solution satisfies for any $p<k$
$$
\forall \, t \in [0,T], \quad \| f(t,\cdot) - f_{\var}(t,\cdot)
\|_{H^p_{per}} \,\,\leq \,\, \bar \varphi(\var)
$$
where $f(t,\cdot)$ is the solution of unperturbed periodized Boltzmann
equation (\ref{unperturb}), and $\bar \varphi(\var)$ is another
explicit function which goes to zero as $\var$ goes to zero. Hence up
to reducing $\var$, the perturbed solution remains close to the
unperturbed solution on the finite time interval on which we have
constructed it.
\end{itemize} 
\end{proposition}

\begin{proof} %[Proof of Proposition~\ref{lem:bdedtime}] 
%Let $T>0$ be the arbitrary length of the fixed time interval. 
We set 
\begin{equation}
\label{Mkt}
M_k(T) := \|f_0\|_{H^k_{per}}\, e^{\C\,M\,T}.
\end{equation} 
First, applying Proposition~\ref{prop1} we have proven that there exists a small $\bar\tau>0$ such that the perturbed Boltzmann equation (\ref{perturb}) admits an unique smooth solution on the time interval $[0,\bar\tau]$ with
$$
\|f_\var(t) \|_{L^1} \,\leq \, M
$$
and
$$
\|f_\var(t) \|_{H^k_{per}} \,\leq \, M_k(\bar\tau) \,\leq M_k(T).
$$
Moreover from Lemma~\ref{lmm4.4}, there exist $\hat\tau\leq\bar\tau$ and $\hat\var>0$, only depending on $M$, $M_k(\bar\tau)<M_k(T)$, $R$ and the collision kernel $B$ such that for all $0<\var<\hat\var$,
$$
\forall \, v  \in \mathcal{D}_L, \quad f_\var(\hat\tau,v) > 0,
$$
$$
\| f^-(t)\|_{L^\infty} \,\leq\, \eta(\var), \quad t\in[0,\hat\tau].  
$$
and
$$ 
\forall \, t \in [0,\hat \tau], \quad \| f(t) \|_{H^k _{per}} \,\leq \|f_0\|_{H^k_{per}}\, e^{\C\,M\,2\,\hat\tau}. 
$$

Then, from the preservation of mass under the action of $Q^R$ and $P_\var$:
$$
\int_{\mathcal{D}_L} Q^{R}(f_\var,f_\var)(v)\,dv\,=\, \int_{\mathcal{D}_L} P_{\var}(f_\var)(v)\,dv\,=\, 0,
$$
we have that
$$
\int_{\mathcal{D}_L} f_\var(\hat\tau,v)\,dv \,=\,\int_{\mathcal{D}_L} f_0(v)\,dv. 
$$
Since $f_0$ is a non negative function, it gives that at time $\hat\tau$
$$
\| f_\var(\hat\tau) \|_{L^1} \,=\,\| f_0 \|_{L^1},
$$
and on the time interval $t\,\in\,[0,\hat\tau]$ we have
$$
\| f_\var(t) \|_{L^1} \,\leq\, M \quad{\rm and}\quad \| f_\var(t) \|_{H^k_{per}} \,\leq\, M_k(T).
$$
Therefore, we consider the perturbed Boltzmann equation (\ref{perturb}) starting from $f_\var(\hat\tau)$ as initial data. On the time interval $[\hat\tau,2\,\hat\tau]$, we apply Proposition~\ref{prop1} and get that
$$
\|f_\var(t)\|_{L^1} \,\leq\, 2 \, \|f_\var(\hat\tau) \|_{L^1} \,=\,  2 \, \|f_0 \|_{L^1}\,=\,M, \quad \forall\,t\,\in\,[\hat\tau,2\,\hat\tau] 
$$ 
and
\begin{eqnarray*}
\|f_\var(t)\|_{H^k_{per}}  &\leq& \|f_\var(\hat\tau)\|_{H^k_{per}}\, e^{\C\,M\,(2\,\hat\tau\,-\,\hat\tau)} 
\\
&\leq& \|f_0\|_{H^k_{per}}\, e^{\C\,M\,2\,\hat\tau}  
\\
&\leq& \|f_0\|_{H^k_{per}}\, e^{\C\,M\,T} \,=\, M_k(T),\quad \forall\,t\,\in\,[\hat\tau,2\,\hat\tau]. 
\end{eqnarray*}
Moreover, since $\hat\tau$ only depends on $M$, $M_K(T)$, $B$ and $R$, we can again apply Lemma~\ref{lmm4.4} on the time intervall $[\hat\tau,2\,\hat\tau]$, which yields that
$$
f_\var(2\,\hat\tau) >0.
$$
We finally proceed by induction to prove existence and uniqueness of a smooth solution $f_\var$ of the perturbed Boltzmann equation (\ref{perturb}) on the time interval $[0,T]$, which proves assertions $(i)$, $(ii)$ and $(iii)$.

To prove (iv), we compute the difference between the solution $f(t)$ to the unperturbed 
problem \eqref{unperturb} and the solution $f_\var(t)$ to \eqref{perturb}:
$$
\frac{\partial (f - f_\var)}{\partial t} = \frac{1}{2}\left( Q^R(f-f_\var,f+f_\var) + Q^R(f+f_\var,f-f_\var)\right)  + P_\var(f_\var).
$$
Then, using the smoothness of $f$ and $f_\var$, we have  from Lemma~\ref{lmm4.1} for any $p<k$ 
$$
\|Q^R(f-f_\var,f+f_\var) \|_{H^p_{per}}, \,\, \|Q^R(f+f_\var,f-f_\var) \|_{H^p_{per}}\,\,\leq\,\, \C_p(M)\,\|f+f_\var\|_{H^p_{per}} \,\|f-f_\var\|_{H^p_{per}} 
$$ 
and since the perturbation is small (assumption \eqref{hyp3:P})
$$
\| P_\var(f) \|_{H^p_{per}}\,\, \le\,\, \varphi(\var),
$$
it yields  that for all $t\,\in [0,T]$
$$
\|f(t)-f_\var(t)\|_{H^p_{per}} \,\leq  \bar \varphi(\var)
$$
for some function $\bar \varphi(\var)$ going to zero as $\var$ goes to zero.
\end{proof} 

\section{Asymptotic behavior and global in time stability}\label{sec5}
\setcounter{equation}{0}

In this section we shall study the asymptotic behavior of the (unperturbed) 
periodized Boltzmann equation \eqref{unperturb} based on a regularity study (in the spirit of \cite{MV:04}) and 
the entropy -- entropy production theory (mainly relying on the method developed 
in \cite{Vill:cerc}). Finally on the basis of these results 
we shall prove a global in times stability result for the perturbed equation \eqref{perturb}.

\subsection{Regularity study of the periodized Boltzmann equation} 

Let us prove the following result 
\begin{proposition} 
\label{prop:reg} 
Let us consider $0 \le f_0 \in L^1(\cal{D}_L)$ such that $f_0 \in H^k_{per}(\cal{D}_L)$ for some 
$k \ge 0$. Then there is a constant $\C > 0$ depending on the $L^1$ and $H^k_{per}(\mathcal{D}_L)$ norms 
of $f_0$ such that the unique global non-negative solution $(f(t))_{t \ge 0}$ to the periodized 
equation (\ref{unperturb}) satisfies 
$$
\forall \, t \ge 0, \quad \| f(t) \|_{H^k_{per}} \le \C. 
$$ 
\end{proposition}

We proceed as in \cite{MV:04}. In particular we shall 
extend Lions regularity result on $Q^+$ to the truncated case $Q^{R,+}$ \cite{lions1, lions2}. 
Hence we shall first prove the regularity property on the gain operator when the collision kernel is smooth 
and compactly supported, avoiding cancellations at zero relative velocities. 
Then, we shall include the non-smooth part of the kernel using the loss operator. 

% Let us consider a collision kernel  (\ref{HSkernel}) and the regularized function $\chi^R_\eta:\R\rightarrow \R^+$
% $$
% \chi^R_\eta(z) = \left\{
% \begin{array}{ll}
% \phi_s\left(\frac{\eta-|z|}{\eta}\right) & \textrm{if } |z| \leq\eta,
% \\
% 1 & \textrm{if } \eta \leq |z| \leq R,
% \\
% \phi_s\left(\frac{|z|-R}{\eta}\right) & \textrm{if } R \leq |z| \leq R + \eta,
% \\
% 0 & \textrm{else} 
% \end{array}\right.
% $$
% where $\phi_s$ is a smooth function such that $\phi_s(0)=1$, $\phi_s(1)=0$ and 
% $\phi_s^{(k)}(0)=\phi^{(k)}_s(1)=0$ for all $k\geq 1$. 
% Also, let $\Theta_\eta:[-1,1] \rightarrow \R$ be a $C^\infty$ function such that
% $$
% \Theta_\eta(u) = \left\{
% \begin{array}{ll}
% 1 & \textrm{if } -1+2\eta\leq u \leq 1-2\eta,
% \\
% \phi_s\left(\frac{-1+2\eta-u}{\eta}\right) & \textrm{if } -1 + \eta \leq u \leq -1+2\eta,
% \\
% \phi_s\left(\frac{u-1+2\eta}{\eta}\right) & \textrm{if } 1 - 2\eta \leq u \leq 1-\eta,
% \\
% 0 & \textrm{else} 
% \end{array}\right.
% $$
We shall split the collision kernel into a smooth and a non-smooth
part. As a convention, we shall use subscripts ``$s$'' for smooth and
``$ns$'' for the non-smooth parts.  In terms of the classical
truncation (\ref{eq:QRclass}) we set
\begin{equation*}
\left\{
\begin{array}{l}  
\mathcal{B}^R_s(|g|,\cos \theta)= \mathcal{B} (|g|,\cos
  \theta)\,\,\chi^R_\eta(|g|)\, \, \zeta_\eta(|g|) \, \, 
  \Theta_\eta\left(\cos \theta \right),
  \vspace{0.2cm} \\
  \mathcal{B}^R_{ns}(|g|,\cos \theta) = \mathcal{B}^R(|g|,\cos
  \theta) - \mathcal{B}^R_s(|g|,\cos \theta),
\end{array}
\right.
\end{equation*}
where $\chi^R _\eta(g)$ is the mollified $C^\infty$ version of ${\bf
  1}_{|g| \le R}$, $\zeta_\eta(g)$ is the mollified $C^\infty$ version of
${\bf 1}_{|g|\ge \eta}$, and $\Theta_\eta$ is a $C^\infty$ function on
$[-1,1]$ which is $1$ on $-1+2\eta\leq u \leq 1-2\eta$, and $0$ in
$[-1,-1+\eta)$ and $(1-\eta,1]$ (the parameter $\eta$ is the
mollification parameter).

In terms of the ``fast'' truncation (\ref{eq:QRfast}) we set
\begin{equation*}
  \left\{
    \begin{array}{l}
  \mathcal{B}^R_s(|y|,|z|) = \mathcal{B}
  (|y|,|z|)\,\,\chi^R_\eta(|z|)\,\, 
  \chi^R_\eta(|y|)\, \, \zeta_\eta(|z+y|) \, \,  
  \Theta_\eta\left(\frac{|y|}{\sqrt{|y|^2+|z|^2}}\right),
  \vspace{0.3cm}\\
  \mathcal{B}^R_{ns}(|y|,|z|) = 
  \mathcal{B}^R(|y|,|z|) - \mathcal{B}^R_s(|y|,|z|),
  \end{array}
  \right.
\end{equation*}
with the same notations.

We deduce the following decomposition of the collision operator:
$$
Q^{R,+} = Q^{R,+}_s + Q^{R,+}_{ns},
$$
where (for instance, with the variables from the ``fast'' truncation)
\begin{equation}
\label{smooth:Qp}
Q^{R,+}_s(f,f) = \int_{y\in \R^d}\int_{z\in \R^d} \mathcal{B}^R_s (|y|,|z|)\,\delta(y \cdot z)\, f(v+z) \,f(v+y)  \,dy\, dz.
\end{equation}
Under the assumption that both $\Phi$ and $b$ defined in
(\ref{HSkernel}) are smooth, the regularized truncature introduced
above ensures that there exist two functions $\Phi^R_\eta$ and
$b^R_\eta$ such that
\begin{equation}
\label{smooth:ker}
\left\{ 
\begin{array}{l}
B^R_\eta(|z|,\cos\theta)=\Phi_\eta^R(|z|)\,b^R_\eta(\cos\theta) \vspace{0.3cm} \\ 
\Phi^R_\eta \in C^\infty_0(\R^d \backslash \{0\}), 
\quad b^R_\eta \in C^\infty_0(-1,1). 
\end{array} 
\right.
\end{equation}
%and
%$$
%\mathcal{B}^R_s(|y|, |z|) =  2^{d-1} \, B^R_\eta\left(\sqrt{|y|^2+|z|^2},\frac{|y|}{\sqrt{|y|^2+|z|^2}}\right) \,\left(|y|^2+|z|^2\right)^{-\frac{d-2}{2}}.
%$$
In the following lemma we shall prove the regularity property of $Q^{R,+}_s$.
\begin{lemma}
  Let $B^R_\eta(|v-v_\star|,\cos\theta)$ satisfy the assumption
  (\ref{HSkernel})-(\ref{hyp:theta}) and (\ref{smooth:ker}). Then, for
  all $r\in \R^+$
$$
\big\| Q^{R,+}_s(f,f) \big\|_{H^{r+\frac{d-1}{2}}_{per}}\ \ \leq\ \
C_{reg}(r,\mathcal{B}^R_s) \ \|f\|_{L^1} \ \|f\|_{H^r_{per}},
$$
where the constant $C_{reg}(r,\mathcal{B}^R_s)$ only depends on $r$
and on the collision kernel.
\end{lemma}
\begin{proof}
  We closely follow the proof given by Lions \cite{lions1,lions2} and
  simplified and then reformulated in \cite{Wenn:94,MV:04}. Again the
  preservation of the translation invariance by the truncation is
  fundamental.  Starting from the collision operator in the form
  (\ref{smooth:Qp}) and performing a change of variable we get for
  $v\in\mathcal{D}_L$
$$
Q^{R,+}_s(f,f) = \int_{\R^d \times \R^d}
\tilde{\mathcal{B}}^R_s(|v^\prime_\star-v^\prime|,|v^\prime-v|)\,\delta\left((v^\prime_\star-v)
  \cdot (v^\prime-v)\right)\,f(v_\star^\prime) \,f(v^\prime)
\,dv^\prime_\star\, dv^\prime,
$$
where $\tilde{\mathcal{B}}_s^R$ only depends on $\mathcal{B}_s^R$.
Then we set \cite{wennberg}
$$
Tg(y) = \int_{y+y^\perp} \tilde{\mathcal{B}}_s^R(|z|,|y|) \,g(z) \,dz, \quad \tau_z g(\cdot) = g(\cdot-z), 
$$
where $$
y^\perp = \left\{ z \in \R^d, \quad z^t\cdot y = 0\right\}
$$
and easily get for $v\in\mathcal{D}_L$
$$
Q^{R,+}_s(f,f) \ \ = \ \ \int_{\R^d} f(v^\prime)\,\left(\tau_{v^\prime} \circ T \circ \tau_{-v^\prime}\right)f(v) \, dv^\prime.
$$ 
Now, we want to estimate Sobolev norms of $Q^{R,+}_s$ as a function
defined in the torus $\mathcal{D}_L$. Applying the Fubini theorem with
the discrete and Lebesgue measures and the Cauchy-Schwarz inequality,
it leads to estimate the Sobolev norms of the Radon transform $T$ on
the torus $\mathcal{D}_L$
\begin{equation}
\label{res1}
\big\| Q^{R,+}_s(f,f) \big\|_{H^{r+\frac{d-1}{2}}_{per}}^2 \ \ \leq \ \ \|f\|_{L^1} \,\int_{\R^d} |f(v^\prime)|\,\left\|\, \tau_{v^\prime} \circ T \circ \tau_{-v^\prime}\, f \,\right\|_{H^{r+\frac{d-1}{2}}_{per}}^2 \,dv^\prime.
\end{equation}

On the one hand, since the kernel $\tilde{\mathcal{B}}_s^R$ is
compactly supported in $y$ and $z$, the operator $T$ maps periodic
functions $g$ to a compactly supported function $Tg$ with $\supp\left(
  Tg \right) \subset B_R \subset\mathcal{D}_L$. Then, we can consider
$Tg$ as a function in the whole space $\R^d$ .

On the other hand, using the regularized truncations $\chi_\eta^R$ and
$\Theta_\eta$ and the smoothness of $\Phi$ and $b$ in
(\ref{HSkernel}), it yields that
$B^R_\eta=\Phi_\eta^R(|v-v_\star|)\,b_\eta^R(\cos\theta)$ with
$$
\Phi_\eta^R\in C^\infty_o(\R), \quad b_\eta^R \in C^\infty_o(-1,1). 
$$
Then, we can directly apply the result in \cite{MV:04}, where the
authors proved the following regularity estimates on the Radon
transform $T$ for smooth kernels
$$
 \| Tg \|_{H^{r+\frac{d-1}{2}}} \ \ \leq \ \ C_{reg}(r,{\mathcal{B}}_s^R) \, \| g\|_{H^r}
$$
for a function $g$ defined in $\R^d$. However in the proof of the
latter inequality, we can replace $g$ by the smooth and compactly
supported function $g\,\chi_\eta^R$ for which
$\supp\left(\chi_\eta^R\,g\right) \subset B(0,R)\subset
\mathcal{D}_L$. Thus, for all $g$ defined in the torus
$\mathcal{D}_L$, we get
\begin{equation}
\label{res2}
\| Tg \|_{H^{r+\frac{d-1}{2}}_{per}}\ \ =\ \| Tg \|_{H^{r+\frac{d-1}{2}}}\ \ \leq \ \ C_{reg}(r,\mathcal{B}_s^R)  \, \| g\|_{H^r_{per}}.
\end{equation}
Finally gathering (\ref{res1}) and (\ref{res2}), we obtain the result
$$
\big\| Q^{R,+}_s(f,f) \big\|_{H^{r+\frac{d-1}{2}}_{per}} \ \ \leq \ \
C_{reg}(r,{\mathcal{B}}_s^R) \|f\|_{L^1} \, \| f\|_{H^r_{per}}.
$$
\end{proof}
\begin{corollary} \label{cor1} Let $B^R_\eta(|v-v_\star|,\cos\theta)$
  satisfy the assumption (\ref{HSkernel})-(\ref{hyp:theta}) and
  (\ref{smooth:ker}). Then, for all $p\in (1,\infty)$
$$
\| Q^{R,+}_s(f,f) \|_{L^q}\ \ \leq \ \ C_{reg}(r,\mathcal{B}^R_s)\
\|f\|_{L^1} \ \|f\|_{L^p},
$$
with
$$
q = \left\{
\begin{array}{ll}
  \displaystyle{\frac{p}{2-\frac{1}{d} + p\left(\frac{1}{d}-1\right)}} & \textrm{ if } p \in(1,2]
  \\
  \displaystyle{p\,d }  & \textrm{ if } p \in[2,\infty).
\end{array}\right.
$$
\end{corollary}
\begin{proof}
  It is a direct consequence of Sobolev embedding and interpolation
  between $L^p$ spaces.
\end{proof}
Now we extend the regularity of $Q^{R,+}$ to general non-smooth
kernels.
\begin{lemma}
  Let $B$ be a collision kernel satisfying
  (\ref{HSkernel})-(\ref{hyp:theta}). Then, for all $p>1$, there exist
  constants $C$, $\kappa$ and $q<p$ ($q$ only depending on $p$ and
  $d$), such that for all $\delta>0$, and for all measurable function
  $f$
  $$
  \big\|Q^{R,+}(f,f)\big\|_{L^p}\ \ \leq\ \ C\,\delta^{-\kappa} \, \|f
  \|_{L^1}\,\|f\|_{L^q} \ +\ \delta \, \|f \|_{L^1}\,\|f\|_{L^p}.
  $$  
\end{lemma}
\begin{proof}
  We use a decomposition approach and split the operator $Q^{R,+}$ as
  the sum of a smooth part and a non-smooth part
  $$
  Q^{R,+} = Q^{R,+}_s + Q^{R,+}_{ns},
  $$
  where $Q^{R,+}_s$ is given by (\ref{smooth:Qp}). Then, applying
  Corollary \ref{cor1} we have for all $p\in (1,\infty)$, there exist
  $q<p$, namely (the role of $q$ and $p$ are exchanged here with
  respect to Corollary \ref{cor1})
  $$
  q = \left\{
    \begin{array}{ll}
      \displaystyle{\frac{(2d-1)p}{d+(d-1)p}} & \textrm{ if } p \in(1,2d]
      \\
      \displaystyle{\frac{p}{d} }  & \textrm{ if } p \in[2d,\infty).
    \end{array}\right.
  $$
  and $C_{reg}(\eta,\mathcal{B}_s^R)$, depending on the the
  regularization parameter $\eta$ and blowing-up polynomially when
  $\eta\rightarrow 0$ such that
  \begin{equation}
    \label{egal:1}
    \big\| Q^{R,+}_s(f,f) \big\|_{L^p}\ \ \leq \ \ C_{reg}(\eta,\mathcal{B}^R_s)\ \|f\|_{L^1} \ \|f\|_{L^q}.
  \end{equation}
  Now, we need to estimate the remainder
  $Q^{R,+}_{ns}=Q^{R,+}-Q^{R,+}_s$. To this aim, we split it as
  $$
  Q^{R,+}_{ns}(f,f) =  Q^{R,+}_{1}(f,f) + Q^{R,+}_{2}(f,f) +
  Q^{R,+}_{3}(f,f) + Q^{R,+} _4 (f,f),
  $$
  with
  \begin{eqnarray*}
    Q^{R,+}_1(f,f) &=&  \int_{\R^d\times\R^d}
    \mathcal{B}(|y|,|z|)\,\delta(y \cdot z)\,\chi^R(|z|)\,
    \left[\chi^R(|y|)-\chi^R_\eta(|y|)\right]\, \zeta_\eta \,
    \Theta_\eta \, f^\prime\,f^\prime_\star\,dy\,dz,
    \\
    Q^{R,+}_2(f,f) &=&  \int_{\R^d\times\R^d}
    \mathcal{B}(|y|,|z|)\,\delta(y \cdot z)\,\chi^R_\eta(|y|)\, 
    \left[\chi^R(|z|)-\chi^R_\eta(|z|)\right]\, \zeta_\eta \,
    \Theta_\eta \, f^\prime\,f^\prime_\star\,dy\,dz,\\
   Q^{R,+}_3(f,f) &=&  
    \int_{\R^d\times\R^d} \mathcal{B}(|y|,|z|)\,\delta(y \cdot
    z)\,\chi^R_\eta(|z|)\,\chi_{\eta}^R(|y|) \, \zeta_\eta \, \left[1-\Theta_\eta\left(\frac{|y|}{\sqrt{|y|^2+|z|^2}}\right)\right]\, f^\prime\,f^\prime_\star\,dy\,dz \\
   Q^{R,+}_4(f,f) &=&  
    \int_{\R^d\times\R^d} \mathcal{B}(|y|,|z|)\,\delta(y \cdot
    z)\,\chi^R_\eta(|z|)\,\chi_{\eta}^R(|y|) \, \Theta_\eta \,
    (1-\zeta_\eta) \, f^\prime\,f^\prime_\star\,dy\,dz.
  \end{eqnarray*}
  On the one hand, we give a first estimate in $L^1(\mathcal{D}_L)$
  applying directly the estimate in Lemma \ref{lmm4.0}:
  \begin{equation}
    \label{n1}
    \| Q^{R,+}_{\alpha} \|_{L^1} \ \ \leq \ \  C_1(R,\mathcal{B}) \,\| f\|_{L^1} \| f\|_{L^1},\quad \alpha \in\{1,2,3,4\}.
  \end{equation}
  On the other hand, we treat for instance the operator
  $Q^{R,+}_1(f,f)$ and have for a fixed $v\in\mathcal{D}_L$
  \begin{eqnarray*}
    \lefteqn{\left|Q^{R,+}_1(f,f)(v)\right| }
    \\
    &\leq&  \int_{\R^d\times\R^d} \mathcal{B}(|y|,|z|)\,\delta(y \cdot z)\,\chi^R(|z|)\,\left|\chi^R(|y|)-\chi^R_\eta(|y|)\right|\, |f^\prime|\,|f^\prime_\star|\,dy\,dz
    \\
    &\leq&  \int_{\mathcal{D}_L} f^\prime\, \left(\int_{\R^d} \mathcal{B}(|v-v^\prime_\star|,|v-v^\prime|)\,\delta((v-v^\prime_\star)\cdot (v-v^\prime))\, \right.
    \\
    && \quad \times \left. \left|\chi^R(|v-v^\prime_\star|)-\chi^R_\eta(|v-v^\prime_\star|)\right|\, |f^\prime_\star|\,dv^\prime_\star\,\right) dv^\prime
    \\
    &\leq & \|f\|_{L^1}\, \| f\|_{L^\infty} \,\sup_{y\in \R^d} \left( \int_{z \in  \ y^\perp} \mathcal{B}^R(|z|,|y|) \,\left|\chi^R(|z|)-\chi^R_\eta(|z|)\right|\, dz\right)
    \\
    &\leq & C_\infty(R,B)\,\eta\,   \|f\|_{L^1}\, \| f\|_{L^\infty}.
  \end{eqnarray*}
  Using similar techniques, we prove that for $\alpha\in\{1,2,3\}$
  \begin{equation}
    \label{ninfini}
    \| Q^{R,+}_{\alpha} \|_{L^\infty} \, \, \leq \, \,  C_\infty(R,\mathcal{B}) \,\eta \,\| f\|_{L^1} \| f\|_{L^\infty}.
  \end{equation}
  For the fourth term, we have using the cancellation of the collision
  kernel $B$ at small relative velocities as $|v-v_*|^\gamma$:
  \begin{equation}
    \label{4infini}
    \| Q^{R,+}_{4} \|_{L^\infty} \, \, \leq \, \,  C_\infty(R,\mathcal{B}) \,\eta^\gamma \,\| f\|_{L^1} \| f\|_{L^\infty}.
  \end{equation}
  Finally, by the Riesz-Thorin interpolation Theorem, from
  (\ref{ninfini}-\ref{4infini}) and (\ref{n1}), we deduce that for
  $p\in[1,+\infty]$ there exist $\C_p(R,B)>0$ and $\beta \in (0,1]$
  such that
  \begin{equation}
    \label{egal:2}
    \| Q^{R,+}_\alpha(f,f) \|_{L^{p}}\,\, \leq\,\, \C_p(R,B) \,\eta^\beta \| f \|_{L^1} \, \| f \|_{L^p}.
  \end{equation}
  To sum up we have obtained for all $p>1$ and $\eta>0$, there exist
  $C>0$, $q<p$, $\kappa_0>0$ and $\beta\in (0,1)$ such that
  $$
  \|Q^{R,+}(f,f)\|_{L^p}\, \, \leq\, \, C \, \eta^{-\kappa_0} \, \|f
  \|_{L^1}\,\|f\|_{L^q} \, \, +\, \, \eta^\beta \, \|f
  \|_{L^1}\,\|f\|_{L^p}.
  $$
  The conclusion follows by choosing $\eta$ small enough.
 \end{proof}

\begin{proof}[Proof of Proposition \ref{prop:reg}]
  Now the proof of the propagation of regularity bounds is done
  exactly as in \cite[Section 4 and Subsections 5.1 \& 5.2]{MV:04}
  (except for the simplification that there is no moments estimates to
  take care of).
\end{proof}

\subsection{Entropy -- entropy production inequalities}

The periodized equation (\ref{unperturb}) preserves non-negativity,
and for a non-negative distribution $f$ one can formally compute an
$H$ theorem (see \cite{FMP}):
$$
\frac{d}{dt} H(f(t)) = - D(f(t)) \le 0 
$$
with 
$$
H(f) = \int_{\cal{D}_L} f \, \log f \, dv 
$$
and 
$$
D(f) = - \int_{\cal{D}_L} Q^R (f,f) \, \log f \, dv = \frac14 \,
\int_{\cal{D}_L \times \cal{C}_R} (f^\prime f^\prime_\star - f
f_\star) \, \log \left(\frac{f^\prime f^\prime_\star}{f
    f_\star}\right) \, \cal{B}(y,z) \, dv \, dy \, dz.
$$

Then we can state the result which relates the entropy functional $H$
and the entropy production functional $D$:
\begin{proposition}
\label{prop:eep}
We consider the periodized Boltzmann collision operator for some
truncation parameter $R> \sqrt 2 \, L$, and we assume that the
collision kernel satisfies $B \ge b_0 \, |v-v_\star|^\gamma$, $\gamma
>0$, for $|v-v_\star| \le R$. Then for any $\eta, \alpha >0$ there is
$k \in \N$ and $M, K >0$ (depending only on $\eta, \alpha,b_0, \gamma,
R$) such that
$$
D(f) \ge K \, H(f | m_\infty)^{1+\eta}, \qquad m_\infty =
\frac{\rho}{|\cal{D}_L|}, \quad \rho= \int_{\cal{D}_L} f \, dv,
$$
for any $\alpha \le f \in L^1(\cal{D}_L)$ with
$H^k_{per}(\mathcal{D}_L)$ norm bounded by $M$.
\end{proposition}

\Remarks 

1. Note that this is a functional inequality independent of the flow
of the Boltzmann equation itself.  \smallskip

2. In the case of the classical Boltzmann equation with $v \in \R^d$
the entropy production functional controls the relative entropy
according to the Maxwellian equilibrium. Here the equilibrium is a
constant, defined by the mass of $f$ divided by the total volume of
the torus.
%\smallskip

%3. Note that unlike the classical Boltzmann equation, we can prove a
% {\em linear} control of the entropy production functional over the
% relative entropy, and a superlinear ``almost linear" control as
% obtained in the papers of Villani for instance. This seems to
% confirm the idea advanced by Villani that the ``defect" of entropy
% production close to equilibrium has to do with what happens in the
% large velocities (not present here since we are in the torus \dots).
\medskip

We shall adapt the method developed in \cite[Proof of
Theorem~2.1]{Vill:cerc}.  In the first step, we treat the case of a
collision kernel $B$ which is uniformly bounded from below. In this
case we prove the equivalent of the so-called Cercignani conjecture in
the context of the Boltzmann operator periodized in the velocity
space.

\Remark Note that the assumption $R \ge \sqrt 2 \, L$ allows to
replace in the bound from below, when needed, the truncation by the
integration over the whole torus.

\medskip

\begin{lemma}\label{lem:cerc} 
  Let us consider a collision kernel $B$ which satisfies $B \ge b_0
  >0$, a truncation $Q^R$ defined by (\ref{eq:QRclass}) or
  (\ref{eq:QRfast}), together with $R \ge \sqrt 2 \, L$. Then there is
  an explicit constant $K$ such that for any $0 \le f \in
  L^1(\cal{D}_L)$ we have
  $$
  D(f) \ge K \, H(f | m_\infty).
  $$
\end{lemma} 
\begin{proof}%[Proof of Lemma~\ref{lem:cerc}] 
  We proceed in several steps.  \smallskip

  \noindent {\bf Step 1.} Since the entropy production functional is
  monotonous in terms of the collision kernel $\cal{B}$, it is no
  restriction to replace $\cal{B}$ by $1$ in the sequel for the
  estimate from below. Moreover it is always possible to bound from
  below the truncation $|v-v^\prime| \le R$ and $|v-v^\prime_\star|
  \le R$ (in case we performed the truncation for the fast spectral
  method) by the classical truncation $|v-v_\star| \le R$.  \smallskip

  \noindent {\bf Step 2.} Using Jensen's inequality on the sphere
  integration (coming back to the classical truncation by the previous
  remark) and the joint convexity of the function $(X,Y) \mapsto (X-Y)
  \, (\log X - \log Y)$ on $\R_+ \times \R_+$ we compute
  $$
  D(f) \ge \C \, \int_{v \in \cal{D}_L, \ v_\star \in B(v,R)}(F-G) \, \log \frac{F}{G} \, dv \, dv_\star =: \bar D(f)
  $$
  where $F = f\, f_\star$ and 
  $$
  G = \frac{1}{|\ens{S}^{d-1}|}\, \int_{\ens{S}^{d-1}} f^\prime \, f^\prime_\star \, d\sigma. 
  $$ 
  Let us study more precisely the function $G$. As it was already
  observed by Boltzmann himself, the function $G$ only depends on
  $v+v_\star$ and $(|v|^2 + |v_\star|^2)/2$. Moreover, here it is also
  periodic on the torus $\cal{D}_L$ since $f$ is periodic. It implies
  (when $f$ is smooth, but we can always use mollifications to relax
  this assumption here) that it in fact only depends on $v+v_\star$.
  \smallskip

  \noindent {\bf Step 3.} Let us denote by $S_t$ the semigroup of the
  heat equation on $L^1(\cal{D}_L)$ (and for brevity we keep the same
  notation for its semigroup in $L^1(\cal{D}_L ^2)$). Then the
  semigroup is compatible with the symmetries in the sense that:
  $$
  S_t ( f \, f_\star) \,=\, (S_t f) \,\, (S_t f_\star)
  $$
  and $S_t G$ only depends on $v+v_\star$ (this follows from a
  straightforward computation using the explicit formula for the Green
  kernel of $S_t$).  \smallskip

  \noindent
  {\bf Step 4.} Then we have the following computation as in \cite{Vill:cerc}: 
  $$
  \frac{d}{dt} \Bigg|_{t=0}\Bigg[ S_t \left( (F -G) \, \log \frac{F}{G} \right) - (S_t F - S_t G) \, \log \frac{S_t F}{S_t G} \Bigg] 
  = \left| \frac{\nabla F}F - \frac{\nabla G}G \right|^2 \, (F+ G),
  $$
  where $\nabla$ denotes the gradient with respect to $(v,v_\star)\in
  \R^{2d}$.

  Then we bound from below the truncation ${\bf 1}_{B(v,R)}(v_\star)$
  by the integration over the whole torus for $v_\star$ (since $R$ is
  large enough), and we compute
  \begin{eqnarray*}
    - \frac{d}{dt} \Bigg|_{t=0} \bar D (S_t f) &=& 
    \int_{\cal{D}_L\times\cal{D}_L}\left| \frac{\nabla F}F - \frac{\nabla G}G \right|^2 \, (F+ G) \, dv \, dv_\star
    \\
    && - \int_{\cal{D}_L\times \cal{D}_L}\Delta \left( (F -G) \, \log \frac{F}{G} \right) \, dv \, dv_\star 
    \\
    &=& \int_{\cal{D}_L\times\cal{D}_L}\left| \frac{\nabla F}F - \frac{\nabla G}G \right|^2 \, (F+ G) \, dv \, dv_\star.
  \end{eqnarray*}
  We deduce by the semigroup property that for all $t>0$
  \begin{equation*}
    - \frac{d}{dt} \bar D (S_t f)  \ge 
    \int_{\cal{D}_L\times\cal{D}_L}\left| \frac{\nabla S_t F}{S_t F} 
      - \frac{\nabla S_t G}{S_t G} \right|^2 \, (S_t F+ S_t G) \, dv \, dv_\star
  \end{equation*}
  and therefore
  $$
  \bar D(f) \ge \int_0 ^{+\infty} \left( \int_{\cal{D}_L\times\cal{D}_L}\left| \frac{\nabla S_t F}{S_t F} 
      - \frac{\nabla S_t G}{S_t G} \right|^2 \, (S_t F+ S_t G) \, dv \, dv_\star \right) \, dt.
  $$
  \smallskip
  
  \noindent {\bf Step 5.} We now use the fact that the operator
$$
P : 
\left\{
\begin{array}{lll}
\R^d\times \R^d &\mapsto& \R^d
\vspace{0.2cm} \\
(A,B) &\mapsto& (A-B)
\end{array}
\right. 
$$
is bounded from $\cal{D}_L \times \cal{D}_L$ to $\cal{D}_L$. Hence
$$
\left| \frac{\nabla S_t F}{S_t F} - \frac{\nabla S_t G}{S_t G} \right|^2 \ge C_1 \, \left| \frac{P \nabla S_t F}{S_t F} \right|^2
$$
since $P \nabla S_t G =\nabla_v G - \nabla_{v_\star} G = 0$ from the
fact that $G$ only depends on $v+v_\star$.  We deduce
$$
\bar D(f) \ge C\,\int_0 ^{+\infty} \left( \int_{\cal{D}_L\times
    \cal{D}_L}\left| \frac{\nabla_v S_t f}{S_t f} - \left(
      \frac{\nabla_{v_\star} S_t f}{S_t f} \right)_\star \right|^2 \,
  (S_t f \, S_t f_\star + S_t G) \, dv \, dv_\star \right) \, dt
$$
and thus (dropping the term $S_t G$)
$$
\bar D(f) \ge C\,\int_0 ^{+\infty} \left( \int_{\cal{D}_L\times
    \cal{D}_L}\left| \frac{\nabla_v S_t f}{S_t f} - \left(
      \frac{\nabla_{v_\star} S_t f}{S_t f} \right)_\star \right|^2 \,
  S_t f \, S_t f_\star \, dv \, dv_\star \right) \, dt.
$$
\smallskip

\noindent {\bf Step 6.} From now on the proof departs slightly more
from \cite{Vill:cerc}: it is simpler since we are in the torus and we
have more symmetries.  Let us show the following functional
inequality: for any smooth non-negative function $h$,
$$
\int_{\cal{D}_L\times\cal{D}_L}\left| \frac{\nabla_v h}{h} - \left(
    \frac{\nabla_{v_\star} h}{h} \right)_\star \right|^2 \, h \,
h_\star \, dv \, dv_\star\ge C_2 \, I(h | m_\infty)
$$
where 
$$
I(h | g) = \int_{\cal{D}_L} h \, \left|\nabla_v \log \frac{h}{g} \right|^2 \, dv.
$$
The proof only amounts to Jensen's inequality on the variable
$v_\star$: since $\int_{\cal{D}_L} h\,dv=\rho$,
$$
\int_{\cal{D}_L\times\cal{D}_L}\left| \frac{\nabla_v h}{h} 
- \left( \frac{\nabla_{v_\star} h}{h} \right)_\star \right|^2 \, h\, h_\star \, dv \, dv_\star \ge \frac{C}{\rho} \,  
\int_{\cal{D}_L}\left| \int_{\in \cal{D}_L} \left( \frac{\nabla_v h}{h} 
- \left( \frac{\nabla_{v_\star} h}{h} \right)_\star \right) \, h_\star \, dv_\star \right|^2 \, h \, dv.
$$
Then as 
$$
\int_{\cal{D}_L}\nabla h_\star \, dv_\star \,=\,0
$$
we deduce 
$$
\int_{\cal{D}_L\times\cal{D}_L}\left| \frac{\nabla_v h}{h} 
- \left( \frac{\nabla_{v_\star} h}{h} \right)_\star \right|^2 \, h \, h_\star \ge C  \,  \int_{\cal{D}_L}\left| \frac{\nabla_v h}{h} \right|^2 \, h \, dv \,=\, \C \, I(h|m_\infty).
$$
\smallskip

\noindent 
{\bf Step 7.}  So far we have proved 
$$
D(f) \ge \bar D(f) \ge C_3 \, \int_0 ^{+\infty} I(S_t f | m_\infty) \, dt.
$$
Then a trivial computation shows that 
$$
\frac{d}{dt} H(S_t f | m_\infty) = - I(S_t f | m_\infty).
$$
Moreover from the explicit formula for $S_t$ we have 
$$
H(S_t f | m_\infty) \xrightarrow[]{t \to +\infty} 0
$$
and thus we finally obtain 
$$
D(f) \ge \bar D(f) \ge C_3 \, \left( H(S_0 f | m_\infty) - 0 \right) \ge C_3 \, H( f | m_\infty).
$$
\end{proof}

Now we are ready to prove Proposition~\ref{prop:eep}.  
%We consider a
%general truncation parameter $R$ and we assume that $B \ge b_0 \,
%|v-v_\star|^\gamma$, $\gamma >0$ for $|v-v_\star|\le R$.  Indeed 
Since we deal with a bounded velocity domain we do not care about
possible decay of the collision kernel at large relative velocity (as
for soft potentials) and the only cancellation we have to treat is for
zero relative velocities. 
%Moreover the main model we want to cover is
%the hard spheres case where $B(|z|) = |z|$ up to a constant.

\begin{proof}[Proof of Proposition~\ref{prop:eep}]
We only mention the difference as compared to the previous proof. 

The reduction to a collision kernel uniformly bounded from below
studied in Lemma \ref{lem:cerc} is done as in
\cite{Vill:cerc}[Theorem~4.1]: one write for some small $\delta >0$
$$
B(|v-v_\star|) \,\,\ge\,\, \delta^\gamma \, \left( B_0 - {\bf 1}_{B(0,\delta)}(|v-v_\star|) \right) 
$$
where $B_0 \,\,\ge\,\, b_0 \,>\,0$ and we deduce 
$$
D(f) \,\,\ge\,\, \delta^\gamma \, \left( D_0(f) \,-\, \tilde D_\delta(f) \right) 
$$
where $D_0$ is the entropy production functional corresponding to $B_0$, and 
$$
\tilde D_\delta \,=\, \frac{1}{4} \, \int_{\cal{D}_L\times\cal{D}_L}\int_{\ens{S}^{d-1}}
(f^\prime f^\prime _\star - f f_\star) \, \log \frac{f^\prime f^\prime _\star}{f f_\star} \, 
{\bf 1}_{B(0,\delta)}(|v-v_\star|)  \,d\sigma\,dv\,dv_\star. 
$$

Then we have the following Lemma, which is proved exactly similarly as
\cite[Theorem~4.2]{Vill:cerc}. It is even simpler since Maxwellians
are replaced by constant functions and the study of the tail is not
needed (we omit the proof for brevity).
\begin{lemma} \label{lem:diag}
For any $\var \in (0,1)$ and $\alpha >0$, there are constants of smoothness $k,M$ and some 
corresponding constant $C_{diag} >0$ such that 
$$
\tilde D_\delta \le C_{diag} \, H(f | m_\infty)^{1-\var} \, \delta^{d/4}
$$
for any $\alpha \le f \in L^1(\mathcal{D}_L)$ with $H^k_{per}(\mathcal{D}_L)$ norm bounded by $M$. 
\end{lemma}

But since we have
$$
D_0(f) \ge \C \, H(f|m_\infty) 
$$
from Lemma \ref{lem:cerc}, it is straightforward to get the result by
choosing correctly the parameter $\delta$.

\end{proof}

Now we can proceed to the proof of Theorem~\ref{th:existence}.
\subsection{Proof of the global in time stability}

In this subsection we shall turn to the question of obtaining uniform
bounds as well as global existence, in order to conclude the proof of
Theorem~\ref{th:existence}.  Indeed in Section~4 the smallness
assumption on the truncation parameter $\var$ {\it a priori} depends
on $T$ and could go to $0$ as $T$ goes to infinity, since it depends
on regularity bounds growing exponentially in times.

In order to overcome this difficulty, we shall combine the following
arguments:
\begin{itemize}
\item for the unperturbed problem \eqref{unperturb} we have a Liapunov
  structure and the solution converges to a unique prescribed
  equilibrium from the regularity and entropy production studies;
\item the equilibrium distribution of the unperturbed problem
  \eqref{unperturb} (that is the constant functions on the torus) are
  also equilibrium distribution of the perturbed problem
  \eqref{perturb};
\item by taking the size of the perturbation small enough (measured in
  terms of $\var$) it is possible to construct a solution to the
  perturbed problem on an arbitrarily large time interval $[0,T]$, on
  which moreover the perturbed solution remains close to the
  unperturbed solution \eqref{unperturb}, say in Sobolev norms;
\item finally the constant equilibrium functions are non-linearly
  stable for the perturbed problem, with a stability domain
  independent on the size of the perturbation.
\end{itemize}

Hence we shall deduce that as soon as the time for which the perturbed
solution departs from the unperturbed solution is larger that the
time-scale of relaxation to equilibrium for the unperturbed problem
\eqref{unperturb}, the perturbed solution shall be trapped by the
stability domain of the equilibrium before instability due to the
perturbation can develop. Let us formalize these arguments in a proof:

\begin{proof}[Proof of the global stability and asymptotic behavior in Theorem~\ref{th:existence}] 
Let us consider some initial datum $0 \le f_0 \in L^1$ which belongs
to $H^k_{per}(\mathcal{D}_L)$, and some smooth balanced perturbations
$f_{0,\var}$ of it. These perturbations have the same mass, and
therefore are corresponding to the same equilibrium (this is the
reason for this assumption). 

Since the smooth balanced perturbation preserves the constant
equilibrium of (\ref{unperturb}) and is nonlinearly locally stable in
any $H^k_{per}(\cal{D}_L)$, we fix a $\eta >0$ such that the constant
distribution $m_\infty \,=\,\rho/|{\cal D}_L|$ associated to the mass
$\rho$ of $f_0$ on the torus has attraction domain with size $\eta$ in
$H^k_{per}(\mathcal{D}_L)$ for the perturbed problem (\ref{perturb}).

On the one hand, we show that there exists a unique solution $f(t)$ to
(\ref{unperturb}) and from Propositon \ref{prop:reg} we obtain uniform
regularity bounds for all $t\geq 0$
$$
 \| f(t) \|_{H^k_{per}} \le \C.
$$ 
Moreover, from Proposition \ref{prop:eep} there exists a time $T_0$
such that the solution $f(t)$ is $\eta/2$-close to the equilibrium in
$H^p_{per}(\mathcal{D}_L)$ ($p<k$) for $t\geq T_0$ (using the
Csisz\'ar-Kullback inequality in the torus, see
\cite[Theorem~1]{Bolley-Villani} for instance):
$$
\| f(t) - m_\infty \|_{H^p _{per}} \,\,\leq \,\, \eta/2.
$$ 

On the other hand, applying Propostion~\ref{lem:bdedtime} with
$T=T_0$, we prove that there exists $\hat \var$, which only depends on
the $H^k_{per}(\mathcal{D}_L)$ and $L^1(\mathcal{D}_L)$ norms of $f_0$
such that for all $\var$ such that $0<\var< \hat\var$, there exists a
unique smooth solution $f_\var(t) \in H_{per}^k(\mathcal{D}_L)$ to
(\ref{perturb}) on $[0,T_0]$ such that
$$
\forall \, t \in [0,T_0], \quad \| f_\var(t) \|_{H^k_{per}} \,\, \leq\,\, \C(T_0),
$$
and (for any $p<k$) 
$$ 
\| f(t) - f_{\var}(t) \|_{H^p_{per}} \,\,\leq \,\, \bar \varphi_{T_0}(\var),
$$
where $f(t)$ is solution to (\ref{unperturb}) and $\bar
\varphi_{T_0}(\var)$ goes to zero when $\var$ goes to zero.

Then, we fix a perturbation parameter $\hat \var$ small enough such
that for $\var \in (0, \hat \var)$ the perturbed solution $f_\var$
satisfies
%can
%be constructed up to the time $T_0$ at least, and remains
%$\eta/2$-close to the unperturbed solution \eqref{unperturb} on this
%time interval
$$
\| f(t) - f_{\var}(t) \|_{H^p_{per}} \,\,\leq \,\, \eta/2.
$$
Finally, at time $T_0$ the perturbed solution $f_\var$ belongs to the
stability domain of the constant distribution for the perturbed
problem and it is trapped. 

Therefore there exists a unique global smooth solution $f_\var$, which
is uniformly bounded for all $t\geq 0$ and such that for any $p<k$
$$
\| f_\var(t) \|_{H^p_{per}} \,\, \leq\,\,    \max\left(\C(T_0),\C+\eta\right).
$$
This achieves the proof of $(i)$, $(ii)$, $(iii)$ and $(iv)$.

\end{proof}

\section{Application: stability and convergence of spectral methods}\label{sec6}
\setcounter{equation}{0}

In this section we consider the following spectral approximation of
(\ref{unperturb})
\begin{equation*}
%\label{BE:perturb}
\frac{\partial f_N}{\partial t} =\mathcal{P}_N\,Q^R(f_N,f_N),
\end{equation*}
where $\mathcal{P}_N$ denotes the orthogonal projection on $\PP_N$ in
$L^2(\mathcal{D}_L)$ (the space of trigonometric polynomials with
degree less at most $N$ in each direction).

The goal of this section is to prove the following theorem:
\begin{theorem}
\label{th:stab} 
Consider any non negative initial datum $f_0 \in
H^k_{per}(\mathcal{D}_L)$, with $k> d/2$, which is not zero
everywhere. Then there exists $N_0\in\NN$ (depending on the mass and
$H^k_{per}(\mathcal{D}_L)$ norm of $f$) such that for all $N\geq N_0$:

\begin{itemize}
\item[$(i)$] there is a unique global solution $f_N=f_N(t,\cdot)$ to the following problem
\begin{equation}
\label{spec:sol}
\left\{\begin{array}{l}
\displaystyle{\frac{\partial f_N}{\partial t}  \,=\, \mathcal{P}_N\,Q^R(f_N,f_N),}
\\
\,
\\
f_N(t=0) = \mathcal{P}_N f_0 ;
\end{array}\right.
\end{equation}
\item[$(ii)$] for any $p<k$, there exists $\C>0$ such that 
$$
\forall \, t \ge 0, \quad \| f_N (t, \cdot) \|_{H^k_{per}} \,\,\leq\,\,\C; 
$$ 

\item[$(iii)$] this solution is everywhere positive for time large
  enough, and the mass of its negative values can be made uniformly
  (in times) $L^\infty$ small as $N \to \infty$;

\item[$(iv)$] this solution $f_N$ converges to $f(t)$ the solution to
  (\ref{unperturb}) with the spectral accuracy, uniformly in time;
   
\item[$(v)$] this solution converges exponentially fast to a constant
  solution on the torus prescribed by the mass conservation law.
\end{itemize} 
\end{theorem}

%Before proving Theorem~\ref{th:stab}, 

To prove Theorem~\ref{th:stab}, we want to apply
Theorem~\ref{th:existence} with the perturbation
$$
P_N^R(f_N) :=   \mathcal{P}_N\,Q^R(f_N,f_N) \,-\,Q^R(f_N,f_N),
$$
which preserves the mass:
$$
\int_{{\cal D}_L} P_N^R(f_N) \,dv \,=\, \int_{{\cal D}_L}
\left( \mathcal{P}_N\,Q^R(f_N,f_N) \,-\,Q^R(f_N,f_N) \right) \,dv\,=\,0,
$$ 

In the next Lemma, we prove a consistency and smoothness result for
this approximation.

\begin{lemma} 
  Consider a non negative function $f \in H_{per}^k(\mathcal{D}_L)$,
  with $k> d/2$, which is not zero everywhere. Then, there exists $\C
  >0$ depending only on the collision kernel $B$ and the truncation 
  such that for all $p\,\in\,[0,k]$ we have
  \begin{equation}
    \label{res:cons:1} 
    \| P_N^R(f) \|_{H^p_{per}} \,\,\leq\,\, \C\, \|f\|_{L^1}\, \| f\|_{H^p_{per}}.
  \end{equation}
  Moreover, %there exists a function $\varphi(N)$ such that $\varphi(N)/rightarrow 0$, as $N$ goes to infinity and
  for all $p\,\in\,[0,k]$
  \begin{equation}
    \label{res:cons:2} 
    \| P_N^R(f) \|_{H^p_{per}} \,\,\leq\,\, \C \,\|f\|_{L^1}\, \frac{\| f\|_{H^k_{per}}}{N^{k-p}}.
  \end{equation}
  \label{lmm5.2}
\end{lemma}
\begin{proof}
First, we split the operator $P_N$ as
$$
\| P_N^R(f) \|_{H^p_{per}} \leq \| Q^R(f,f) \|_{H^p_{per}}  +  \| \mathcal{P}_N\,Q^R(f,f) \|_{H^p_{per}}.
$$
As in the proof of Lemma~\ref{lmm4.1} we get that for all $p\,\in [0,k]$
\begin{equation}
\label{resu:0}
\|Q^R(f,f)\|_{H^p_{per}}^2 \,\leq\,  C_p(R,B)\,\, \|f\|_{L^1}^2 \, \| f \|_{H^{p}_{per}}^{2}.
\end{equation}
Concerning the interpolation error estimate, the following result
holds. If $u\in H^p_{per}(\mathcal{D}_L)$ for some $p\geq 1$, then 
\begin{equation}
\label{error:est}
\| u - \mathcal{P}_N u \|_{H^p_{per}}  \leq \frac{\C}{N^{k-p}}\,\| u \|_{H^{k}_{per}}.
\end{equation}
Then, taking $p=k$ in the latter inequality and from (\ref{resu:0}) we
obtain
\begin{equation}
\label{resu:1}
\|\mathcal{P}_N Q^R(f,f)\|_{H^p_{per}}^2 \,\leq\, \|Q^R(f,f)\|_{H^p_{per}}^2 \,\leq\, C_p(R,B)\,\, \|f\|_{L^1}^2 \, \| f \|_{H^{p}_{per}}^{2}.
\end{equation}
Gathering (\ref{resu:0}) and (\ref{resu:1}), we finally get
$$
\| P_N^R(f) \|_{H^p_{per}}  \,\,\leq\,\,  C_p(R,B)\,\, \|f\|_{L^1} \, \| f \|_{H^{p}_{per}}.
$$
Moreover, using again the error estimate (\ref{error:est}), it yields
\begin{eqnarray*}
\| P_N^R(f) \|_{H^p_{per}}  &\leq& \C \,\frac{\|Q^R(f,f)\|_{H^k_{per}}}{N^{k-p}} 
\\
&\leq & C(R,B)\,\, \|f\|_{L^1} \, \frac{\| f \|_{H^{k}_{per}}}{N^{k-p}}.
\end{eqnarray*}
\end{proof}

%\subsection{Spectral study of the perturbed equation}
 
Let us now perform a linearized study of the perturbed equation
(\ref{spec:sol}) by classical Fourier-basis decomposition.  The only
equilibrium distributions of the equation (\ref{unperturb}) are the
constant, prescribed by the mass conservation. Let us consider the
linearized version of the perturbed equation \eqref{spec:sol} around
such a constant $m_\infty$:
$$
\frac{\partial f}{\partial t}= m_\infty \, \cal{L}^{N,R} (f)
$$
where 
$$
\cal{L}^{N,R} (f) = \cal{P}_N \, \left[ Q^R(f,1) + Q^R(1,f) \right]. 
$$

Let us prove the following lemma:
\begin{lemma} 
\label{prop:spec}
The operator $\cal{L}^{N,R}$ is bounded and self-adjoint in
$L^2(\cal{D}_L)$. Moreover it is non-negative, its null space is given
by the constant functions, and it has a spectral gap $\lambda >0$.  As
a consequence, the constant are nonlinearly locally stable in any
$H^k_{per}(\cal{D}_L)$ for the equation \eqref{spec:sol}, with a
stability domain independent on $N$.
\end{lemma}

\begin{proof}%[Proof of Lemma~\ref{prop:spec}] 
  The boundedness is trivial. Then, the periodized operator $Q^R$ is
  translation invariant, which implies that the Fourier modes
  $$
  e_k(v) = \frac{\exp \left[ i \, \frac{\pi}L \, ( k \cdot v ) \right]}{| \cal{D}_L |} 
  $$
  for $k \in \Z^d$ are trivially eigenfunctions of $\cal{L}^{N,R}$. This provides 
  a complete orthonormal eigenbasis in $L^2(\cal{D}_L)$. A trivial 
  computation yields 
  $$
\cal{L}^{N,R} (e_k) = a_k \, e_k \, {\bf 1}_{|k| \le N} \ \ \mbox{with } \  
a_k := - \int_{\cal{C}_R} \big[ 1 + e_k(y+z) - e_k(y) - e_k(z) \big] \, \cal{B}(y,z) \, dy \, dz.
$$  
In particular we deduce that $\cal{L}^{N,R} = \cal{L}^{N,R} \, \cal{P}_N$, 
and the self-adjointness comes from the following identity obtained 
by the usual changes of variables: 
\begin{multline*}
\int_{\cal{D}_L} \cal{L}^{N,R}(f) \, g \, dv 
= - \frac14 \, 
\int_{\cal{D}_L \times \cal{C}_R} \left[ \cal{P}_Nf ' +  \cal{P}_Nf^\prime_\star -  \cal{P}_Nf -  \cal{P}_Nf_\star \right] \\ 
    \times  \left[  \cal{P}_Ng^\prime +  \cal{P}_Ng^\prime_\star -  \cal{P}_Ng -  \cal{P}_Ng_\star \right] \, \cal{B}(y,z) \, dy \, dz \, dv.
\end{multline*}

Another formula for $a_k$ is readily deduced from the previous
representation:
$$
a_k = - \frac14 \, \int_{\cal{D}_L \times \cal{C}_R} \big|
(e_k)^\prime + (e_k)^\prime_\star - (e_k) - (e_k)_\star \big|^2 \,
\cal{B}(y,z) \, dy \, dz \, dv.
$$

One sees from the second representation that $a_k = a_{-k} \le 0$ for
any $k \in \Z^d$, and from the first representation and Lebesgue
theorem it is easily seen that for $|k| \to \infty$ the coefficients
$a_k$ converge as
$$
a_k \xrightarrow[|k|\to \infty]{} a_\infty =  - \int_{\cal{C}_R} \cal{B}(y,z) \, dy \, dz \ \in \, (-\infty,0).
$$

Hence we deduce that $a_k \in [a_\infty,0]$ for any $k \in \Z^d$, with
asymptotic convergence towards $a_\infty$ for large $k$. Moreover the
null space can be computed: for some smooth periodic function $\phi$,
the equation
$$
\phi^\prime + \phi^\prime_\star - \phi - \phi_\star = 0
$$
implies that the third-order derivative of $\phi$ is zero, and the
periodicity then imposes that it is a constant.  Thus the null space
is spanned by $e_0$. It concludes the proof of the
existence of a spectral gap
$$
\lambda_N := \min \big\{ |a_k|, \ k \in [|-N,N|]^d, \ k \not = 0 \big\} >0
$$
which is uniformly bounded from below as $N \to +\infty$, since 
$$
\lambda_N \to \lambda_\infty := \min \big\{ |a_k|, \ k \in \Z^d, \ k
\not = 0 \big\} >0.
$$

The non-linear stability in $L^2$ comes from the fact that for the perturbation 
$h = f - m_\infty$, we have the following control on the bilinear part: 
$$
\| \cal{P}_N Q^R (h,h) \|_{H^k_{per}} \le C_{\cal{B},R} \, \| h \|^2 _{L^2} 
$$
for some given constant $C_{\cal{B},R}>0$ independent on $N$. 

Finally using the eigenbasis of the Fourier modes, a similar study can
be performed in any Sobolev space $H^k_{per}(\cal{D}_L)$.
\end{proof} 

\Remark Exact computations could be made for particular physical
collision kernels $\cal{B}$ (in a similar way as the computation of
the kernel modes $\beta(l,m)$ in \cite{FiRu:FBE:03,FiRu:04,MP:03,FMP}).

\subsection{Proof of Theorem \ref{th:stab}}
Consider the numerical solution $f_N$ given by solving
(\ref{spec:sol}). We can formulate the problem as a perturbation of
the truncated Boltzmann equation. Indeed setting
$$
P_N(f_N) = -(Id - \mathcal{P}_N) Q^R(f_N,f_N), 
$$  
the problem (\ref{spec:sol}) can be written as
$$
\frac{\partial f_N}{\partial t}  = Q^R(f_N,f_N) + P_N(f_N).
$$
Then, applying Lemma~\ref{lmm5.2}, the perturbation $P_N$ satisfies
the assumptions of Theorem~\ref{th:existence}. 
%$$
%\| P_N^R(f_N) \|_{H^p_{per}} \,\,\leq\,\, \C\, \|f_N\|_{L^1}\, \| f_N\|_{H^p_{per}}.
%$$
%and
%$$
%\| P_N^R(f_N) \|_{H^p_{per}} \,\,\leq\,\, \C\, \|f_N\|_{L^1}\, \| f_N\|_{H^p_{per}}.
%$$
Moreover since $f_0\,\in H^k_{per}(\mathcal{D}_L)$, we have
straightforwardly
$$
\| f_N(0) \|_{H^p_{per}} \leq \| f_0\|_{H^p_{per}}, \quad \| f_N(0) -
f_0 \|_{H^p _{per}} \to 0.
$$
Hence, we can directly apply Theorem~\ref{th:existence} to the
perturbation $P_N$, which proves that there exists $N_0$ large enough
and only depending on $f_0$, the kernel $B$ and the truncation, such
that for all $N\geq N_0$, the perturbed system admits a unique
uniformly smooth solution, which converges to a constant, and
satisfies all the points in Theorem~\ref{th:existence}.

\bigskip

%%-----------------------------
%%      your bibliography
%%-----------------------------

\end{document}